\documentclass{amsart}
\usepackage{graphicx}
\usepackage{amssymb}
\usepackage{amsfonts}
\swapnumbers
\newtheorem{theorem}{Theorem}[section]
\newtheorem{lemma}[theorem]{Lemma}
\newtheorem{corollary}[theorem]{Corollary}
\newtheorem{proposition}[theorem]{Proposition}
\theoremstyle{definition}

\newtheorem{remark}[theorem]{Remark}

\numberwithin{equation}{section}
 \theoremstyle{plain}    
 
 \numberwithin{equation}{section} 
 \numberwithin{figure}{section} 
 \theoremstyle{plain}    
 \theoremstyle{plain}    
 \theoremstyle{remark}    
 \newtheorem*{acknowledgement*}{Acknowledgement} 

\newcommand{\cB}{{\mathcal B}}

\newcommand{\cD}{{\mathcal D}}
\newcommand{\cE}{{\mathcal E}}
\newcommand{\cF}{{\mathcal F}}

\newcommand{\cH}{{\mathcal H}}

\newcommand{\cL}{{\mathcal L}}

\newcommand{\cN}{{\mathcal N}}

\newcommand{\cR}{{\mathcal R}}

\newcommand{\cX}{{\mathcal X}}
\newcommand{\cY}{{\mathcal Y}}

\newcommand{\te}{{\theta}}

\newcommand{\om}{{\omega}}
\newcommand{\ve}{{\varepsilon}}
\newcommand{\del}{{\delta}}

\newcommand{\gam}{{\gamma}}
\newcommand{\Gam}{{\Gamma}}
\newcommand{\vf}{{\varphi}}

\newcommand{\sig}{{\sigma}}
\newcommand{\al}{{\alpha}}
\newcommand{\be}{{\beta}}

\newcommand{\La}{{\Lambda}}

\newcommand{\bbN}{{\mathbb N}}

\newcommand{\bbZ}{{\mathbb Z}}
\newcommand{\bbI}{{\mathbb I}}

\begin{document}
\title[]{Ergodic Theorems for Nonconventional Arrays and an Extension
of the Szemer\' edi Theorem}%
 \vskip 0.1cm 
 \author{ Yuri Kifer\\
\vskip 0.1cm
 Institute  of Mathematics\\
Hebrew University\\
Jerusalem, Israel}%
\address{
Institute of Mathematics, The Hebrew University, Jerusalem 91904, Israel}
\email{ kifer@math.huji.ac.il}%

\thanks{A part of this work was done during the author's
visit to University of Pennsylvania in Fall of 2016 as the Bogen family 
visiting professor. }
\subjclass[2010]{Primary: 37A30 Secondary: 37A45, 28D05}%
\keywords{Szem\' eredi theorem, multiple recurrence, nonconventional averages,
triangular arrays}%
\dedicatory{  }
 \date{\today}
\begin{abstract}\noindent
The paper is primarily concerned with the asymptotic behavior as 
$N\to\infty$ of averages of
nonconventional arrays having the form $N^{-1}\sum_{n=1}^N\prod_{j=1}^\ell 
T^{P_j(n,N)}f_j$ where $f_j$'s are bounded measurable functions, $T$
is an invertible measure preserving transformation and $P_j$'s are polynomials
of $n$ and $N$ taking on integer values on integers. It turns out that when
$T$ is weakly mixing and $P_j(n,N)=p_jn+q_jN$ are linear or, more generally,
 have the form $P_j(n,N)=P_j(n)+Q_j(N)$ for some integer valued polynomials
  $P_j$ and $Q_j$ then the above averages converge in $L^2$ but for general
polynomials $P_j$ the $L^2$ convergence can be ensured even in the case 
$\ell=1$ only when $T$ is strongly mixing. Studying also
 weakly mixing and compact extensions and relying on Furstenberg's structure
 theorem we derive an extension of Szemer\' edi's theorem saying that for
 any subset of integers $\La$ with positive upper density there exists a 
 subset $\cN_\La$ of positive integers having uniformly bounded gaps such
 that for $N\in\cN_\La$ and at least $\ve N,\,\ve>0$ of $n$'s all numbers
 $p_jn+q_jN,\, j=1,...,\ell,$ belong to $\La$. We obtain
also a version of these results for several commuting transformations which
yields a corresponding extension of the multidimensional Szemer\' edi theorem.

\end{abstract}

\maketitle
\markboth{Yu.Kifer}{Extension of the Szemer\' edi theorem} 
\renewcommand{\theequation}{\arabic{section}.\arabic{equation}}
\pagenumbering{arabic}
\section{Introduction}\label{sec1}\setcounter{equation}{0}

In 1975 Szemer\' edi proved the conjecture of Erd\H os and Turan saying
that any set of integers with positive upper density contains arbitrary long
arithmetic progressions. In 1977 Furstenberg \cite{Fu1} published an ergodic
theory proof of this result which turned out to be a corollary of a multiple
recurrence statement for measure preserving transformations.

Namely, let $(X,\cB,\mu)$ be a probability space, $T:X\to X$ be an invertible 
$\mu$-preserving transformation and $A\in\cB$ be a set of positive 
$\mu$-measure.
Furstenberg proved that in these circumstances for any positive integer $\ell$,
\begin{equation}\label{1.1}
\liminf_{N\to\infty}\frac 1N\sum_{n=1}^N\mu(\bigcap_{j=0}^\ell T^{-jn}A)>0
\end{equation}
which, in fact, implies existence of infinitely many arithmetic progressions
in any set of integers having positive upper density (for a nice exposition
of this result see \cite{FKO}).

An important part of the proof of (\ref{1.1}) was to show that
\begin{equation}\label{1.2}
\frac 1N\sum_{n=1}^N\prod_{j=1}^\ell T^{jn}f_j\stackrel{L^2}
  {\longrightarrow}\prod_{j=1}^\ell\int f_jd\mu\quad\mbox{as}\quad N\to\infty
\end{equation}
(where $T^mf(x)=f(T^mx)$) provided $T$ is a measure preserving weakly mixing
invertible transformation and $f_j$'s are bounded measurable functions. In
fact, (\ref{1.1}) required more general results concerning weak mixing and
compact extensions together with a structure theorem describing all possible
extensions. Observe that in \cite{Be} the $L^2$ convergence (\ref{1.2}) for
weakly mixing transformations was extended from powers $jn$ to arbitrary 
essentially distinct polynomials $P_j(n)$ (i.e. having nonconstant
pairwise differences) taking on integer values on integers.

In this paper we consider the averages of the form
\begin{equation}\label{1.3}
\frac 1N\sum_{n=1}^N\prod_{j=1}^\ell T^{P_j(n,N)}f_j
\end{equation}
where $f_j$'s are bounded measurable functions, $T$ is an invertible measure
preserving transformation and $P_j(n,N),\, j=1,...,\ell,$ are essentially
distinct polynomials of $n$ and $N$ taking on integer values on integers.
It is customary in
probability to call sums whose summands depend on the number $N$ of summands
by the name (triangular) arrays and it seems appropriate to use the same name
for sums in (\ref{1.3}) too while the term "nonconventional" comes from
\cite{Fu3}. 

First, we study the linear case $P_j(n,N)=p_jn+q_jN$ where $p_j$'s are 
distinct and $q_j$'s are arbitrary integers. It turns out that under the
weak mixing assumption on $T$,
\begin{equation}\label{1.4}
  \frac 1N\sum_{n=1}^N\prod_{j=1}^\ell T^{p_jn+q_jN}f_j\stackrel{L^2}
  {\longrightarrow}\prod_{j=1}^\ell\int f_jd\mu\quad\mbox{as}\quad N\to\infty.
  \end{equation} 
In particular, when $\ell=2k,\, q_i=-p_i=k-i+1$ for $i=1,...,k$ and 
$p_i=i-k,q_i=0$ for $i=k+1,...,2k$ the left hand side of (\ref{1.4}) takes on
 the following symmetric form
\begin{equation}\label{1.5}
\frac 1N\sum_{n=1}^N\prod_{i=1}^kT^{i(N-n)}f_{k-i+1}\prod_{i=1}^kT^{in}f_{k+i}.
\end{equation}
It is known by \cite{HK} that when $q_j=0$ for all $j=1,...,\ell$ then the
left hand side of (\ref{1.4}) still converges in $L^2$ also without the
weak mixing assumption on $T$ but not necessarily to the right hand side
of (\ref{1.4}). On the other hand, a simple example shows that for arbitrary
$q_j$'s there is no convergence of the left hand side of (\ref{1.4}) if
$T$ is not weakly mixing. Indeed, take $k=1$ in (\ref{1.5}) and let $T$ be the
rotation of the unit circle by one half of it while $f_1=f_2=f=\bbI_A$ be the 
indicator of an arc $A$ having length less than one half of the circle. Then
$(T^nf)(T^{N-n}f)=\bbI_{T^{-n}A\cap T^{-(N-n)}A}$ and this expression equals
the indicator $\bbI_A$ of $A$ or the indicator $\bbI_{TA}$ of $TA$ (depending
on the parity of $n$) if $N$ is even while it equals zero for otherwise. Thus,
the averages (\ref{1.5}) will be equal to
$\frac 12(\bbI_A+\bbI_{TA})$ for each even $N$ and 0 for each odd $N$.

We complement the above study by considering weak mixing and compact extensions
and relying on the structure theorem from \cite{Fu1} and \cite{Fu2} we conclude
that for any invertible measure preserving transformation $T$, numbers $p_j,q_j$
as above and a set $A$ of positive measure there exists a subset $\cN_A\subset
\bbN$ of positive integers with uniformly bounded gaps, called syndetic set,
 such that
\begin{equation}\label{1.6}
\liminf_{N\to\infty,\, N\in\cN_A}\frac 1N\sum_{n=1}^N\mu(\bigcap_{j=0}^\ell
T^{-(p_jn+q_jN)}A)>0
\end{equation}
while this does not hold true, in general, if we take the limit over all 
positive integers which can be seen from the above example. In fact, we also
show that (\ref{1.6}) follows by a shorter argument relying on recent advanced
 results from \cite{BLL} and \cite{Au} concerning convergencies along F\o lner
 sequences in multidimensionar multiple recurrence results but our direct 
 proof still has a value since, in particular, it concentrates attention on 
 nonconventional arrays and the convergence results like (\ref{1.4}) cannot
 be derived from the above references. 
 
 By the standard Furstenberg's argument (\ref{1.6}) implies an extended version
  of Szemer\' edi's theorem saying that for any subset of integers $\La$ with
  positive upper density there exists a syndetic subset $\cN_{\La}\subset\bbN$
  such that for all $N\in\cN_\La$ and at least $\ve N,\,\ve>0$ of $n$'s, all
  numbers $p_jn+q_jN,\, j=1,...,\ell,$ belong to $\La$. We obtain also more
  general results concerning families of commuting transformations $T_j,\, 
  \hat T_j,\, j=1,2,...,\ell,$ studying the limits of
\[
\frac 1N\sum_{n=1}^N\mu(\bigcap_{j=0}^\ell T_j^{-n}\hat T_j^{-N}A)\,\,\mbox{and of}
\,\,\frac 1N\sum_{n=1}^N\prod_{j=1}^\ell T_j^{n}\hat T_j^{N}f_j
\]
where $A$ and $f_j,\, j=1,...,\ell,$ are as above.

If we consider $P_j(n,N)=P_j(n)+Q_j(N)$ in (\ref{1.3}), where $P_j$'s are
essentially distinct and $Q_j$'s are arbitrary polynomials, then the 
convergence in $L^2$ of nonconventional averages (\ref{1.3}) to the product 
of integrals can be established under the weak mixing assumption. On the other
hand, already for $\ell=1$ and $P_1(n,N)=nN$  weak mixing is not sufficient, 
in general, for the $L^2$ convergence of averages (\ref{1.3}) though strong
mixing suffices here. For $\ell>1$ and general polynomials $P_j(n,N)$ we show
$L^2$ convergence of the expression (\ref{1.3}) assuming strong $2\ell$-mixing
 of $T$.

\begin{acknowledgement*} The author is greatful to anonymous referees for many
 helpful suggestions which led to improvements of the original version of this
  paper. 
\end{acknowledgement*}

\section{Preliminaries and main results}\label{sec2}\setcounter{equation}{0}

Let $(X,\cB,\mu)$ be a separable probability space and $T:X\to X$ be an 
invertible measure $\mu$ preserving transformation. In studying polynomial
nonconventional averages (\ref{1.3}) we start with the linear case
$P_j(n,N)=p_jn+q_jN$. Let $f_i,\, i=0,1,...,\ell,$ be bounded measurable 
functions on $X$. The example described in Introduction shows that, in general,
the limit
 \begin{equation}\label{2.1}
 \lim_{N\to\infty}\frac 1N\sum_{n=1}^N\prod_{j=1}^\ell 
 T^{p_jn+q_jN}f_j
 \end{equation}
 does not exist in the $L^2$-sense. Still, we will
   see that the limit (\ref{2.1}) exists in the $L^2$ sense if $T$ 
   is weakly mixing which means that the product transformation 
   $T\times T$ on $(X\times X,\,\cB\times\cB,\,\mu\times\mu)$ is ergodic
   (see, for instance, \cite{Fu2}). Thus we have the following $L^2$ 
   ergodic theorem for nonconventional arrays.
   
   \begin{theorem}\label{thm2.1}
   Suppose that an invertible transformation $T$ is weakly mixing, $f_j,\,
   j=1,...,\ell,$ are bounded measurable functions and $p_j,\, q_j,\, 
   j=1,...,\ell,$ are integers such that $p_j$'s are distinct (ordered
   without loss of generality as $p_1<p_2<...<p_\ell$) and $q_j$'s
   are arbitrary. Then
   \begin{equation}\label{2.2}
   \lim_{N\to\infty}\frac 1N\sum_{n=1}^N\prod_{j=1}^\ell T^{p_jn+q_jN}
   f_j=\prod_{j=1}^\ell\int f_jd\mu\,\,\quad\mbox{in}\,\, L^2(X,\mu).
   \end{equation}
   \end{theorem}
   The condition of Theorem \ref{thm2.1} that $p_j,\, j=1,...,\ell,$ are
   distinct is important for (\ref{2.2}). Indeed, let 
   $p_1=p_2=p$ and $\int f_1d\mu=0$. Then
   \[
   \frac 1{N}\int\sum_{n=1}^NT^{pn+q_1N}f_1T^{pn+q_2N}f_2d\mu
   =\int f_1T^{(q_2-q_1)N}f_2d\mu,
   \]
   which does not converge to zero as $N\to\infty$, in general, unless
   $T$ is (strongly) mixing (and $q_1\ne q_2$) while under weak mixing
   only convergence ouside of a set of $N$'s having zero density can be
   ensured. We observe (as pointed out by the referee) that Theorem 
   \ref{thm2.1} actually follows from Theorem 3.2 in the recent paper 
   \cite{Ko} though motivation and goals of the latter paper seem to be
   different from ours.
   In fact, we will study convergence in a more general situation of weak
   mixing extensions and, in addition, will consider also compact extensions,
   which together with the structure theorem from \cite{Fu1} and \cite{Fu2} 
   will produce the following result. 
\begin{theorem}\label{thm2.2} Let $p_j,q_j,\, j=0,1,...,\ell,$ be integers
such that $p_0=q_0=0,\, p_j\ne 0$ if $j\ne 0$ and $p_1<p_2<...<p_\ell$. Then
for any $A\in\cB$ with $\mu(A)>0$ there exists an infinite subset of positive
integers $\cN_A\subset\bbN$ with uniformly bounded gaps such that
\begin{equation}\label{2.3}
\liminf_{N\to\infty,\, N\in\cN_A}\frac 1N\sum_{n=1}^N\mu\big(\bigcap_{j=0}^\ell
T^{-(p_jn+q_jN)}A\big)>0.
\end{equation}
\end{theorem}

As the example in Introduction shows this statement does not hold true, 
in general, if we take $\liminf$ over all $N\to\infty$. On the other hand, 
if $q_j=0$ for all $j$ then (\ref{2.1}) was proved in \cite{Fu1} (see also 
\cite{FKO}) with
$\liminf$ over all $N\to\infty$ and it was shown there how such result yields 
the Szemer\' edi type theorem. Recall briefly the latter argument. Let
$\{ 0,1\}^\bbZ=\{\om=(\om_i):\,\om_i\in\{ 0,1\},\,-\infty<i<\infty\}$ be the
space of sequences, $(T\om)_i=\om_{i+1}$ be the left shift and consider the
special sequence $\bar\om=(\bar\om_i)_{i=1}^\infty$ where $\bar\om_i=1$ if and
 only if $i\in\La)$ with $\La\subset\bbZ$ being a subset of integers with a 
 positive upper density (called, also the upper Banach density), i.e.,
 \begin{equation}\label{2.4}
 \lim_{n\to\infty}\frac {|\La\cap [a_n,b_n)|}{(b_n-a_n)}=d>0
 \end{equation}
 for some sequence of intervals with $b_n-a_n\to\infty$ as $n\to\infty$, 
 denoting by $|\Gam|$ the number of elements in a set $\Gam$. Take $X=$the
 closure in $\{ 0,1\}^\bbZ$ of $\{ T^n\bar\om\}^\infty_{n=-\infty}$ then
 any weak limit $\mu$ of the sequence of measures $\mu_n=(b_n-a_n)^{-1}
 \sum_{j=a_n}^{b_n}\del_{T^j\bar\om}$ (where $\del_\om$ is the unit mass 
 at $\om$) is a $T$-invariant probability measure
 on $X$ and if $A=X\cap\{\om:\,\om_0=1\}$ then $\mu(A)=d>0$. 
 
 It is easy to see that $\La$ contains an arithmetic progression of length 
 $\ell$ if and
 only if $\bigcap_{j=0}^{\ell-1}T^{-jb}A$ is nonempty for some $b\ne 0$. More
 generally, $\La$ contains all numbers $a+p_jn+q_jN,\, j=0,1,...,\ell,$ for
 some $a\in\La$ if and only if $\bigcap_{j=0}^\ell T^{-(p_jn+q_jN)}A$ is nonempty.
 Thus, Theorem \ref{thm2.2} yields the following result.
 
 \begin{corollary}\label{cor2.3} Let $\La$ be a subset of nonnegative integers
 with a positive upper density and $p_j,q_j,\, j=0,1,...,\ell,$ be integers
 satisfying conditions of Theorem \ref{thm2.2}. Then there exist $\ve>0$ and 
 an infinite set of positive integers $\cN_\La$ with uniformly 
 bounded gaps such that for any $N\in\cN_\La$
 the interval $[0,N]$ contains not less than $\ve N$ integers $n$ with the
 property that for some $a_n$,
 \begin{equation}\label{2.5}
 a_n+p_jn+q_jN\in\La\quad\mbox{for all}\quad j=0,1,...,\ell.
 \end{equation}
 In particular, if $\ell=2k$, $q_j=-p_j=k-j+1$ for $j=1,...,k$ and $p_j=j-k,\,
 q_j=0$ for $j=k+1,...,2k$ then for at least $\ve N,\, N\in\cN_\La$ integers
 $n$ in the interval $[0,N]$ the set $\La$ contains arithmetic progressions
 with length $k+1$ of both step $n$ and of step $N-n$.
 \end{corollary}
 
 Clearly, the above corollary does not hold true, in general, if we replace
 $\cN_\La$ by all positive integers. Indeed, let $\La$ be the set of all
 even numbers then $a+n$ and $a+(N-n)$ cannot both belong to $\La$ if $N$ is
 odd since then $a+n$ and $a+(N-n)$ cannot be both even.

   Next, we will discuss an extension of the above results to families of
   commuting transformations, which will yield also a multidimensional version
   of Corollary \ref{cor2.3}. Let $G$ be a multiplicative free finitely
   generated abelian group acting on $X$ by measure $\mu$-preserving 
   transformations which are necessarily invertible. Any such group is 
   isomorphic to a $d$-dimensional integer lattice $\bbZ^d$ group.  
   Let $f_i,\, i=0,1,...,\ell,$ be bounded measurable functions on $X$.
    As in the case of one transformation, in general, the limit
    \begin{equation}\label{2.6}
    \lim_{N\to\infty}\frac 1N\sum_{n=1}^N\prod_{j=1}^\ell T^n_j\hat T^N_jf_j
    \end{equation}
    does not exists if $N\to\infty$ over all $N$. Nevertheless,
     we will see that the limit
    (\ref{2.6}) exists in the $L^2$ sense if the abelian group $G$ is totally
    weak mixing, i.e. it consists of weakly mixing transformations with the
    only exception of the identity.
    
    \begin{theorem}\label{thm2.4} Suppose that distinct and different from
    the identity (id) transformations $T_1,...,T_\ell$ belong to a totally weak
    mixing free finitely generated abelian group $G$ acting on $(X,\cB,\mu)$
     by measure preserving
    transformations. Let $\hat T_1,...,\hat T_\ell$ be invertible 
    $\mu$-preserving transformations of $X$, which commute with each other
    and with $T_1,...,T_\ell$. Then for any bounded
    measurable functions $f_j,\, j=0,1,...,\ell$,
    \begin{equation}\label{2.7}
    \lim_{N\to\infty}\frac 1N\sum_{n=0}^N\prod_{j=1}^\ell T_j^n\hat T_j^N
    f_j=\prod_{j=1}^\ell\int f_jd\mu\,\,\quad\mbox{in}\,\, L^2(X,\mu)
    \end{equation}
    where $T_0=\hat T_0=\mbox{id}$.
    \end{theorem}
    
   Considering weak mixing and primitive extensions we will obtain the
   following generalization of Theorem \ref{thm2.2}.
    
    \begin{theorem}\label{thm2.5} Let $T_j,\,\hat T_j\in G,\, j=1,...,\ell,$
    where $T_1,...,T_\ell$ are distinct and different from the identity id
    of $G$ while $\hat T_1,...,\hat T_\ell$ are any transformations from $G$.
    Then for any $A\in\cB$ with $\mu(A)>0$ there exists an infinite subset of 
    positive integers $\cN_A\subset\bbN$ with uniformly bounded gaps such that
    \begin{equation}\label{2.8}
    \liminf_{N\to\infty,\, N\in\cN_A}\frac 1N\sum_{n=1}^N\mu\big(\bigcap_{j=0}^
    \ell(T_j^n\hat T_j^N)^{-1}A\big)>0
    \end{equation}
    where we set $T_0=\hat T_0=\mbox{id}$.
    \end{theorem}
    
    Clearly, if we set $T_j=T^{p_j}$ and $\hat T_j=T^{q_j}$ then we arrive back 
    at the setup of Theorem \ref{thm2.2}. For $\hat T_j,\, j=1,2,...,\ell,$ 
    equal the identity (\ref{2.9}) was proved in \cite{FK} with $\cN_A=\bbN$
    but our proof will follow more closely Chapter 7 of \cite{Fu2}.
    Similarly to the one transformation 
    case Theorem \ref{thm2.5} yields an extension of a multidimensional
    version of the Szemer\' edi theorem. Recall, the notion of the upper 
    (Banach) density of a set $\La\subset\bbZ^d$. For any two vectors
    $\bar a=(a_1,...,a_d),\, \bar b=(b_1,...,b_d)\in\bbZ^d$ such that 
    $a_i<b_i,\, i=1,...,d,$ denote by $B(\bar a,\bar b)$ the parallelepiped
     $\prod_{i=1}^d [a_i,b_i]$. A set $\La\subset\bbZ^d$ is said to have 
     positive upper (Banach) density if there exists a sequence of 
     parallelepipeds $B(\bar a(n),\bar b(n))$ with $\bar a(n)=(a_1(n),
     ...,a_d(n)),\, \bar b(n)=(b_1(n), ...,b_d(n))$ satisfying 
     $\lim_{n\to\infty}\min_{1\leq i\leq d}(b_i(n)-a_i(n))=\infty$ and
     such that
    \begin{equation}\label{2.9}
    \lim_{n\to\infty}\frac {|\La\cap B(\bar a(n),\bar b(n))|}{\prod_{1\leq i
    \leq d}(b_i(n)-a_i(n))}=d>0
    \end{equation}
    where, again, $|\Gam|$ denotes the number of points in a set $\Gam$.
    
    Since the group in Theorem \ref{thm2.5} is isomorphic to $\bbZ^d$ we can 
    identify the actions of $T_j$ and $\hat T_j$ with additions of some
    vectors $z_i\in\bbZ^d$ and $\hat z_i\in\bbZ^d$. For any ordered finite
    set $\Gam=\{ z_1,...,z_\ell\},\, z_i\in\bbZ^d$, $n\in\bbZ$ and $a\in\bbZ^d$
    we set $n\Gam=\{ nz_1,...,nz_\ell\}$ and $a+\Gam=\{ a+z_1,
    ...,a+z_\ell\}$. Next, if $\Gam=\{ z_1,...,z_\ell\}$ and $\hat\Gam=
    \{\hat z_1,...,\hat z_\ell\},\, z_i,\hat z_i\in\bbZ^d$ are two ordered
    finite sets then we write $\Gam+\hat\Gam=\{ z_1+\hat z_1,...,z_\ell+
    \hat z_\ell\}$. Now Theorem \ref{thm2.5} yields the following extension
    of the multidimensional Szemer\' edi theorem.
    
    \begin{corollary}\label{cor2.6} Let $\La$ be a subset of $\bbZ^d$ with a
    positive upper (Banach) density and let $\Gam=\{ z_1,...,z_\ell\},\,
    \hat\Gam=\{\hat z_1,...,\hat z_\ell\}$ be two ordered sets of vectors
    from $\bbZ^d$ such that $z_1,z_2,...,z_\ell$ are all distinct and non
    zero. Then there exist $\ve>0$ and an infinite set of positive
    integers $\cN_\La$ with uniformly bounded gaps such that for any 
    $N\in\cN_\La$ the interval $[0,N]$ contains not less than $\ve N$
    integers $n$ such that for some $a_n\in\La$,
    \begin{equation}\label{2.10}
    a_n+n\Gam+N\hat\Gam\subset\La.
    \end{equation}
    \end{corollary}
   Corollary \ref{cor2.6} follows from Theorem \ref{thm2.5} similarly to
   the one transformation case. Namely, we consider the action of $\bbZ^d$
   on $\{ 0,1\}^{\bbZ^d}=\{\om=(\om_v),\,\om_v\in\{ 0,1\},\, v\in\bbZ^d\}$
   by $(z\om)_v=\om_{v+z}$ for any $z,v\in\bbZ^d$. Again, we take $X$ to
   be the closure in $\{ 0,1\}^{\bbZ^d}$ of the orbit $\bbZ^d\bar\om$ of the
   special sequence $\bar\om=(\bar\om_v,\,\bar\om_v=1$ if and only if 
   $v\in\La)$ and an $\bbZ^d$-invariant measure $\mu$ comes as a weak limit as 
   $n\to\infty$ of the measures $\prod_{1\leq i\leq d}(b_i(n)-a_i(n))^{-1}
   \sum_{z\in B(\bar a(n),\bar b(n))}\del_{z\bar\om}$ where 
   $B(\bar a(n),\bar b(n)),\, n=1,2,...,$ are the same as in (\ref{2.9}). 
    
   The proofs of the above results proceed similarly to \cite{FKO} and 
   \cite{Fu2},
   and so we will be trying to make a compromise between keeping the paper
   relatively self-contained and still avoiding too many repetitions of 
   arguments from \cite{FKO} and \cite{Fu2}. Though, of course, Theorem
   \ref{thm2.2} is a particular case of Theorem \ref{thm2.5}, in order to
   facilitate the reading, we will consider first the one transformation
   case and then pass to the case of commuting transformations. 
   
   As we mentioned it in Introduction it is possible to give 
   a shorter argument yielding Theorems \ref{thm2.2} and \ref{thm2.5}, 
    which will be presented in Section \ref{sec5}. This argument
   relies on quite general results from the
   recent paper \cite{Au}. In fact, this argument together with \cite{BLL} 
   yields Theorem \ref{thm2.1} with linear terms $p_in+q_iN$ replaced by
   arbitrary polynomials $p_i(n,N),\, i=1,...,\ell,$ taking on integer values
    for integer pairs $n,N$ and such that for any integer $k$ there exist
    $n$ and $N$ with $p_i(n,N)$ divisible by $k$ for each $i=1,...,\ell$.
    The results in \cite{BLL} and \cite{Au} rely on advanced machinery 
    developed with the purpose to derive convergence of nonconventional
    averages in various situations. The direct proof presented here, which
    proceeds along the lines of the original proof in \cite{FKO} and \cite{Fu2},
    still seems to be useful, in particular, for focusing attention on limiting
     behavior of nonconventional arrays, which is a somewhat different point of
   view in comparison to other research on multiple recurrence problems and
    since Theorems \ref{thm2.2} and \ref{thm2.5}  do not follow from \cite{BLL}
   and \cite{Au}.
   
   \begin{remark}\label{rem2.7}
   As we have seen, the limit in Theorem \ref{thm2.1} does not exist, in general, 
 without the weak mixing assumption but it is plausible that the limit may
 exist over syndetic subsequences of $N$'s. It would be interesting also to 
 obtain some uniform versions of Theorems \ref{thm2.2} and \ref{thm2.5} 
 in the spirit of \cite{BHMP}. It would be also natural to find most general
 conditions, which ensure almost everywhere convergence of averages of
 nonconventional arrays though this question is not completely settled even
 for standard nonconventional averages (i.e. without dependence of summands
 on $N$). Finally, we observe that it may be interesting
 to obtain a result of the type
 of Corollary \ref{cor2.3} for the set of primes in place of a set of positive
 upper density extending to this situation the main result of \cite{GT}. 
 In this case relevant sets of $N$'s will probably have gaps containing only 
 bounded number of primes.
\end{remark}

Next, we consider averages of nonconventional arrays (\ref{1.3}) with higher
degree polynomials $P_j(n,N),\, j=1,...,\ell$. When we can separate 
dependencies on $n$ and $N$ the applying the "PET-induction" from \cite{BL}
for polynomials in $n$ and, essentially, treating $N$ as a constant there we
 will obtain in Section \ref{sec6} the following result.
\begin{theorem}\label{thm2.8} Let $T_1,...,T_k$ be different from identity
transformations belonging to a totally weak mixing finitely generated free
abelian group $G$ acting on $(X,\cB,\mu)$ by measure preserving transformations
and $\hat T_1,...,\hat T_k$ be invertible $\mu$-preserving transformations
of $X$ which commute with each other and with $T_1,...,T_k$. Furthermore,
let $P_{ij}(n),\, i=1,...,\ell,\, j=1,...,k,$ be polynomials taking on integer
values on integers and suppose that the expressions
\[
\vf_i(n)=T_1^{P_{i1}(n)}\cdots T_k^{P_{ik}(n)},\, i=1,...,\ell,
\]
and the expressions
\[
\vf_i(n)\vf_j^{-1}(n)=T_1^{P_{i1}(n)-P_{j1}(n)}T_2^{P_{i2}(n)-P_{j2}}\cdots
T_k^{P_{ik}(n)-P_{jk}(n)},\, i,j=1,...,\ell,\, i\ne j,
\]
depend nontrivially on $n$ (i.e. that they are nonconstant maps from $\bbZ$
to $G$). In addition, let $Q_{ij}(N),\, i=1,...,\ell,\, j=1,...,k,$ be 
arbitrary functions of $N$ taking on integer values on integers. Then, for
any bounded measurable functions $f_i,\, i=1,...,\ell,$
\begin{eqnarray}\label{2.11}
&\lim_{N\to\infty}\frac 1N\sum_{n=1}^{N}\prod_{i=1}^\ell T_1^{P_{i1}(n)}\cdots
T_k^{P_{ik}(n)}\hat T_1^{Q_{iq}(N)}\cdots\hat T_k^{Q_{ik}(N)}f_i\\
&=\prod_{i=1}^\ell\int f_id\mu\,\,\quad\mbox{in}\,\, L^2(X,\mu).\nonumber
\end{eqnarray}
\end{theorem}

If all $T_i$'s and $\hat T_i$'s coincide with one transformation $T$ then
(\ref{2.11}) becomes
\begin{equation}\label{2.12}
\lim_{N\to\infty}\frac 1N\sum_{n=1}^N\prod_{j=1}^\ell T^{P_j(n,N)}f_j=
\prod_{j=1}^\ell\int f_jd\mu\,\,\quad\mbox{in}\,\, L^2(X,\mu)
\end{equation}
with $P_j(n,N)$'s taking the form $P_j(n,N)=P_j(n)+Q_j(N)$ where $P_j(n)$'s 
are nonconstant essentially distinct polynomials of $n$ and 
$Q_j(N)$'s are function of $N$, both taking on integer values on integers.
It turns out that for general polynomials of $n$ and $N$ weak mixing may not
 be enough for the $L^2$ convergence in (\ref{1.3}). In Section \ref{sec6} 
we will show employing a version of a spectral argument suggested to us
 by Benji Weiss that already the averages
\begin{equation}\label{2.13}
\frac 1N\sum_{n=1}^NT^{nN}f
\end{equation}
do not converge in $L^2$ as $N\to\infty$, in general, if $T$ is only weak 
mixing. Still, strong mixing of $T$ ensures convergence in $L^2$ for
this example. More generally, we will prove the following result where
we rely on the notion of strong $m$-mixing, which means that 
\begin{eqnarray*}
&\lim_{|k_i-k_j|\to\infty,\,\forall i\ne j}\mu(\bigcap_{i=1}^mT^{-k_i}\Gam_i)\\
&=\lim_{l_1,...,l_{m-1}\to\infty}\mu(\Gam_1\cap T^{-l_1}\Gam_2\cap...\cap 
T^{-(l_1+\cdots +l_{m-1})}\Gam_m)=\prod_{i=1}^m\mu(\Gam_i)
\end{eqnarray*}
for any measurable sets $\Gam_1,...,\Gam_m$.

\begin{theorem}\label{thm2.9} Let $P_j(n,N),\, j=1,...,\ell,$ 
be nonconstant essentially distinct polynomials of $n$ and $N$ 
(i.e. $P_i(n,N)-P_j(n,N),
\, i\ne j$ is not a constant identically) taking on integer values on 
integers and nontrivially depending on $n$ (i.e. $P_i(n,N)$ is not just a
polynomial of $N$). If $T$ is a strongly $2\ell$-mixing invertible 
transformation of $(X,\cB,\mu)$ then (\ref{2.12}) 
 holds true for any bounded measurable functions $f_j,\, j=1,...,\ell$.
\end{theorem}

 Observe that both conditions that the polynomials $P_j$ are essentially 
 distinct and nontrivially depend on $n$ are important for Theorem \ref{thm2.9}
 to hold true. As to the first condition consider
 $\frac 1N\sum_{n=1}^NT^nfT^{n+1}g=\frac 1N\sum_{n=1}^NT^n(fTg)$ which 
 by the $L^2$ ergodic theorem converges as $N\to\infty$ to $\int fTgd\mu$ which
 usually differs from the product of integrals of $f$ and $g$. As to the second 
 condition we can consider $\frac 1N\sum_{n=1}^NT^Nf=T^Nf$ which does not 
 converges at all as $N\to\infty$ unless $f$ is a constant $\mu$-almost 
 everywhere.
 It would be natural to try to show that for $\ell\geq 2$ strong mixing
 (i.e. 2-mixing) is not enough, in general, for Theorem \ref{thm2.9} to hold
 true but this is not easy since then we would have to construct an example of
 a 2-mixing but not $2\ell$-mixing transformation which is a version of the old
 open problem attributed to Rokhlin.
 
 We observe that such dynamical systems as topologically mixing subshifts 
 of finite type, Axiom A diffeomorphisms and expanding transformations
 considered with an invariant Gibbs measure constructed by a H\" older 
 continuous function (potential) are strong mixing of all orders so the
 above theorem is applicable for them. This is also true for the Gauss map
 $Tx=1/x$ mod 1, $x\in(0,1),\, T0=0$ considered with its Gauss invariant 
 measure $\mu(\Gam)=\frac 1{\ln 2}\int_\Gam\frac {dx}{1+x}$, as well as some
 other maps of the interval. Actually, mixing of all orders follows from the 
 property called in probability $\al$-mixing (or strong mixing) and the above
 dynamical systems have this property (and even stronger property called
 $\psi$-mixing with exponential speed, see, for instance, \cite{Bo}, \cite{He}
 and \cite{Bra}). 
 
 These notions are defined via two parameter families of $\sig$-algebras
 $\cF_{mn}\subset\cF$, $-\infty< m\leq n<\infty$ on a probability
 space $(X,\cF,P)$ such that $\cF_{mn}\subset\cF_{m'n'}$ if $m'\leq m\leq n
 \leq n'$. We define also $\cF_{mn}$ for $m=-\infty$ and $n<\infty$, for
 $m>-\infty$ and $n=\infty$ or for $m=-\infty$ and $n=\infty$ as minimal
 $\sig$-algebras containing $\cF_{kn}$ for all $k>-\infty$, containing
 $\cF_{ml}$ for all $l<\infty$ or containing $\cF_{kl}$ for all $-\infty<k\leq
 l<\infty$, respectively. Such family of $\sig$-algebras is called $\al$-mixing
  if
 \[
 \al(n)=\sup_m\{ |\mu(A\cap B)-\mu(A)\mu(B)|:\, A\in\cF_{-\infty,m},\, B\in
 \cF_{m+n,\infty}\}\to 0\,\,\mbox{as}\,\, n\to\infty.
 \]
 Now, we have the following result which is probably well known but for
 readers' convenience we provide details here.
 \begin{proposition}\label{prop2.10} Suppose that $\{\cF_{mn},\,-\infty\leq 
 m\leq n\leq\infty\}$ is an $\al$-mixing family of $\sig$-algebras on a 
 probability space $(X,\cF,\mu)$ with $\cF=\cF_{-\infty,\infty}$. 
 Let $T:X\to X$ be a measure $\mu$-preserving transformation such that
 $T^{-1}\cF_{m,n}\subset\cF_{m+1,n+1}$ for all $m\leq n$. Then for any
 $\Gam_1,...,\Gam_k\in\cF,\, k\geq 2$,
 \begin{equation}\label{2.14}
 \lim_{l_1,...,l_{k-1}\to\infty}\mu(\Gam_1\cap T^{-l_1}\Gam_2\cap...\cap 
 T^{-(l_1+\cdots +l_{k-1})}\Gam_k)=\prod_{i=1}^k\mu(\Gam_i).
 \end{equation}  
 \end{proposition}
  \begin{proof}
  First, observe that if $G_i\in\cF_{m_in_i},\, i=1,...,k$ with $m_i\leq n_i<
  m_{i+1},\, i=1,...,k-1,$ then applying the definition of the mixing 
  coefficient $\al$ subsequently we obtain that
  \begin{equation}\label{2.15}
  \big\vert\mu(\bigcap_{i=1}^kG_i)-\prod_{i=1}^k\mu(G_i)\big\vert\leq
  \sum_{i=1}^{k-1}\al (m_{i+1}-n_i).
  \end{equation}
  Next, let $\Gam_i\in\cF_{mn}$ for some $-\infty<m\leq n<\infty$ and all
  $i=1,...,k$. Since $T^{-1}\cF_{mn}\subset\cF_{m+1,n+1}$ we obtain from 
  (\ref{2.15}) that (\ref{2.14}) holds true. Now, let $\Gam_i\in\cF=
  \cF_{-\infty,\infty},\, i=1,...,k,$ be arbitrary. Then for each $\ve>0$ 
  there exist $m=m(\ve)\leq n=n(\ve)$ and $\hat\Gam_i\in\cF_{mn}$ such that
  $\mu(\Gam_i\triangle\hat\Gam_i)<\ve,\, i=1,...,k,$ where $\triangle$
  denotes the symmetric difference. Since (\ref{2.14}) holds true for
  $\hat\Gam_i$ in place of $\Gam_i\, i=1,...,k,$ we obtain that
  \[
  \limsup_{l_1,...,l_{k-1}\to\infty}\big\vert\mu(\Gam_1\cap T^{-l_1}\Gam_2\cap
  ...\cap T^{-(l_1+\cdots +l_{k-1})}\Gam_k)-\prod_{i=1}^k\mu(\Gam_i)\big\vert
  \leq 2k\ve
  \]
  and since $\ve>0$ is arbitrary (\ref{2.14}) follows.
  \end{proof}
  
  Observe, that a typical application of the above setup is in the symbolic
  setup where $X$ is a sequence space, $T$ is the left shift and the 
  $\sig$-algebras $\cF_{mn}$ are generated by the cylinder sets for which the
  sequence elements on places from $m$ to $n$ are fixed. This can be extended
  to dynamical systems having appropriate symbolic representations via, for
  instance, Markov partitions.

\section{One transformation case}\label{sec3}\setcounter{equation}{0}

In this section we will establish Theorems \ref{thm2.1}, \ref{thm2.2}
and Corollary \ref{cor2.3}.

\subsection{Factors and extensions}\label{subsec3.1}
 The strategy of our proof is the same as
 in \cite{FKO}. It is based on the notions of factors and extensions.
  Recall, that
 if $T$ is a measure preserving transformation of a probability space
 $(X,\cB,\mu)$ and $T^{-1}\cB_1\subset\cB_1\subset\cB$ then $(X,\cB_1,\mu,T)$
 is called a factor of $(X,\cB,\mu,T)$ while the latter is called an
 extension of $(X,\cB_1,\mu,T)$. The latter factor is said to be 
 nontrivial if $\cB_1$ contains sets of measure strictly between 0 and 1.
 It is often more convenient to view factors in the following equivalent
 way (see \cite{FKO} for more details). Namely, the factor $(X,\cB_1,\mu,T)$
  is identified with a system
 $(Y,\cD,\nu,S)$ such that for some measurable onto map $\pi:X\to Y$ we have
 $\pi\mu=\nu$, $\pi T=S\pi$ and $\cB=\pi^{-1}\cD$. Furthermore, $\mu$ 
 disintegrates into $\mu_y,\, y\in Y$ so that $\mu=\int\mu_yd\nu(y)$ and
 $T\mu_y=\mu_{Sy}$ $\nu$-almost everywhere (a.e.).
 
 Next, let $g\in L^2(X,\cB,\mu)$ and let $\cY=(Y,\cD,\nu,S)$ be a factor
 of $(X,\cB,\mu,T)$. Following \cite{FKO} we set
 \[
 E(g|\cY)(y)=\int gd\mu_y.
 \]
 This is essentially the conditional expectation $E(g|\cB_1)$ provided
 $(Y,\cD,\nu,S)$ is identified with $(X,\cB_1,\mu,T)$. Since $\cB_1=
 \pi^{-1}\cD$ and $Y=\pi X$ then $E(g|\cB_1)$ is constant on $\pi^{-1}y$
 for $\nu$-almost all $y$, and so this conditional expectation can be viewed
 as a function
 on $Y$. Since we refer often to \cite{FKO} we will keep the notations from
 there though they differ slightly from the way conditional expectations with 
 respect to $\sig$-algebras are written in probability. We will use also the
 following well known formulas
 \begin{equation}\label{3.1}
 E(gf|\cY)=gE(f|\cY)\,\,\mbox{if}\,\, g\,\,\mbox{is}\,\,\cB_1-\mbox{measurable
 and}\,\, SE(f|\cY)=E(Tf|\cY)
 \end{equation}
 provided $f$ and $fg$ are integrable.
 
 Fix a measure preserving system $(X,\cB,\mu,T)$ and let $\cB_1\subset\cB$ be
  a $T$-invariant $\sig$-subalgebra. If (\ref{2.3}) holds true for any 
  $A\in\cB_1,\,\ell$ and $p_j,q_j,\, j=0,1,...,\ell,$ all satisfying the
  conditions of Theorem \ref{thm2.2} then we say that the action of $T$ on
  the factor $(X,\cB_1,\mu)$ is generalized Szemer\' edi (GSZ). To make this
  shorter we will also say in this case that the action of $T$ on $\cB_1$ is
  GSZ and if $(X,\cB_1,\mu,T)$ is identified with $(Y,\cD,\nu,S)$ then this is
  equivalent to saying that the action of $S$ on $\cD$ is GSZ.
  
  Similarly to \cite{FKO} we can see that the set of factors for which $T$ is 
  GSZ contains a maximal element and that no proper factor can be maximal. The
  proof of Theorem \ref{thm2.2}
  is based on the notions of relative weak mixing and relative compact
  extensions of other factors, which will be defined below. We will show that
  if the action $T$ is GSZ for smaller factor then it is also GSZ for a larger
  factor which is either relative mixing or relative compact with respect to
  the smaller factor. Considered together with two following facts this will
  yield our result. First, similarly to \cite{FKO} we see that if $T$ is GSZ
  for a totally ordered (by inclusion) family of factors $\{\cB_\al\}$ (i.e.
  factors $(X,\cB_\al,\mu)$) then $T$ is GSZ for $\sup_\al\cB_\al$
  (i.e. for $(X,\sup_\al\cB_\al,\mu))$ where the latter is the minimal 
  $\sig$-algebra containing each $\cB_\al$. Secondly, we rely on the general 
  result from \cite{FKO} saying that if $\cX=(X,\cB,\mu,T)$ is an extension of
  $\cY=(Y,\cD,\nu,S)$, which is not relative weak mixing, then there exists an 
  intermediate factor $\cX^*$ between $\cY$ and $\cX$ so that $\cX^*$ is a
  (relative) compact extension of $\cY$.
  
  \subsection{Relative weak mixing}\label{subsec3.2}
   Let $(Z,\cE,\te)$ be a probability space,
  $X=Y\times Z$, $\mu=\nu\times\te$, $\cB=\cD\times\cE$ and $T(y,z)=(Sy,\,
  \sig(y)z)$ where $S:Y\to Y$ preserves a probability measure $\nu$,
  $\sig(y)z$ is measurable in $(y,z)$ and all 
  $\sig(y),\, y\in Y$ preserve the measure $\te$. Then $T$ is measure 
  preserving on $(X,\cB,\mu)$ and $(X,\cB,\mu,T)$ is called in \cite{FKO}
  a skew product of $(Y,\cD,\nu,S)$ with $(Z,\cD,\te)$ (while usually $T$
  inself is called a skew product transformation). Set $\tilde X=Y\times Z
  \times Z$, $\tilde\cB=\cD\times\cE\times\cE$, $\tilde\mu=\nu\times\te
  \times\te$ and $\tilde T(y,z,z')=(Sy,\sig(y)z,\sig(y)z')$. Then $\cX=
  (X,\cB,\mu,T)$ is called a relative weak mixing extension of $\cY=
  (Y,\cD,\nu,S)$ if the action of $\tilde T$ on $(\tilde X,\tilde\cB,
  \tilde\mu)$ is ergodic. 
  
  \begin{proposition}\label{prop3.1}
  Let $(X,\cB,\mu,T)$ be a relative weak mixing extension of $(Y,\cD,\nu,S)$
  and $f_j\in L^\infty(X,\cB,\mu),\, j=0,1,...,\ell$. Then for any $m=1,2,...,
  \ell$,
  \begin{equation}\label{3.2}
  \lim_{N\to\infty}\frac 1N\sum_{n=1}^N\int\big(E(\prod_{j=0}^mT^{p_jn+q_jN}
  f_j|\cY)-\prod_{j=0}^mS^{p_jn+q_jN}E(f_j|\cY)\big)^2d\nu=0
  \end{equation}
  and
  \begin{equation}\label{3.3}
  \lim_{N\to\infty}\bigg\|\frac 1N\sum_{n=1}^N\big(\prod_{j=1}^mT^{p_jn+q_jN}
  f_j-\prod_{j=1}^mT^{p_jn+q_jN}E(f_j|\cY)\big)\bigg\|_{L^2(X,\mu)}=0
  \end{equation}
  where $p_j,q_j,\, j=1,...,\ell,$ satisfy conditions of Theorem \ref{thm2.2}.
  \end{proposition}
  \begin{proof} The proof proceeds similarly to Theorem 8.3 in \cite{FKO}.
  Recall, that the conditional expectations $E(f_j|\cY)$ can 
  be viewed as functions in both $L^\infty(X,\cB,\mu)$ and in $L^\infty(Y,\cD,
  \nu)$, which is identified with $L^\infty(X,\cB_1,\mu)$, and so this 
  conditional expectation is $\cB_1$-measurable. Denote the assertions 
  (\ref{3.2}) and (\ref{3.3}) by $A_m$ and $B_m$, respectively, where both 
  mean that they hold true for all relatively weak mixing extensions of
  $(Y,\cD,\nu,S)$ and all $L^\infty$ functions on corresponding spaces. 
  
   First, observe that $A_0$ is obvious and $B_0$ will not play a role here
   so we can denote by $B_0$ any correct assertion. Next, we proceed by 
   induction in $m$ showing that (cf. \cite{FKO}),
   
   (i) $A_{m-1}$ implies $B_m$ and
   
   (ii) $B_m$ for $(\tilde X,\tilde\cB,\tilde\mu,\tilde T)$ (which is
   also a relative weak mixing extension of $(Y,\cD,\nu,S)$)
    implies $A_m$ for $(X,\cB,\mu,T)$.
   
   We start with (ii) which is easier. If $f_0$ is measurable with respect
   to $\cB_1=\pi^{-1}(\cD)$, the integrals in (\ref{3.2}) have the form
   \begin{eqnarray*}
   &\int f^2_0\big(E(\prod_{j=1}^mT^{p_jn+q_jN}
  f_j|\cY)-\prod_{j=1}^mS^{p_jn+q_jN}E(f_j|\cY)\big)^2d\nu\\
  &\leq\sup|f^2_0|\int\big(E(\prod_{j=0}^{m-1}T^{\tilde p_jn+\tilde q_jN}
  f_j|\cY)-\prod_{j=0}^{m-1}S^{\tilde p_jn+\tilde q_jN}E(f_j|\cY)\big)^2d\nu
  \end{eqnarray*}
  where $\tilde p_j=p_{j+1}-p_1,\,\tilde q_j=q_{j+1}-q_1$ still satisfy
  conditions of Theorem \ref{thm2.2} and we use (\ref{3.1}) here and that 
  $S$ is $\nu$-preserving. Thus $A_m$ follows from $A_{m-1}$ if $f_0$ is 
  $\cB_1$-measurable (assuming the induction hypothesis for all $p_j,q_j$ 
  satisfying the conditions of Theorem \ref{thm2.2}). 
  
  It follows that writing $f_0=(f_0-E(f_0 |\cY))+E(f_0|\cY)$ and using that
  $(a+b)^2\leq 2a^2+2b^2$ we can assume that $E(f_0|\cY)=0$. With this the
  left hand side of (\ref{3.2}) takes the form
  \[
  \lim_{N\to\infty}\int f_0\otimes f_0\big(\frac 1N\sum_{n=1}^N\prod_{j=1}^m
  \tilde T^{p_jn+q_jN}f_j\otimes f_j\big)d\tilde\mu
  \]
  where $g\otimes g(y,z,z')=g(y,z)g(y,z')$ is a function on $\tilde X$ whenever
  $g$ is a function on $X$ (see (6.6) in \cite{FKO}). By $B_m$ for
  $(\tilde X,\tilde\cB,\tilde\mu,\tilde T)$ the above limit equals
  \[
  \lim_{N\to\infty}\int f_0\otimes f_0\big(\frac 1N\sum_{n=1}^N\prod_{j=1}^m
  \tilde T^{p_jn+q_jN}E(f_j\otimes f_j|\cY)\big)d\tilde\mu.
  \]
  Since the sum here is $\cB_1$-measurable we can insert the conditional 
  expectation inside of the integral concluding that the latter limit is
  zero since
  \[
  E(f_0\otimes f_0|\cY)=\int f_0(y,z)f_0(y,z')d\mu_y(z)d\mu_y(z')d\nu(y)
  =\int E(f_0|\cY)^2(y)d\nu(y)=0
  \]
  completing the proof of (ii).
  
  In order to prove (i) we observe that
  \begin{eqnarray*}
  &\prod_{j=1}^mT^{p_jn+q_jN}f_j-\prod_{j=1}^mT^{p_jn+q_jN}E(f_j|\cY)\\
  &=\sum_{j=1}^m(\prod_{i=1}^{j-1}T^{p_in+q_iN}f_i)T^{p_jn+q_jN}(f_j-E(f_j|\cY))
  \prod_{i=j+1}^mT^{p_i+q_iN}E(f_i|\cY).
  \end{eqnarray*}
  It follows that it suffices to prove $B_m$ under the additional condition
  that for some $j_0,\, 1\leq j_0\leq m$ we have $E(f_{j_0}|\cY)=0$ (replacing
  $f_{j_0}$ by $f_{j_0}-E(f_{j_0}|\cY)$).
  
   We now have to show that $\lim_{N\to\infty}\|\psi_N\|_{L^2}=0$ for
  $\psi_N=\frac 1N\sum_{n=1}^N\prod_{j=1}^mT^{p_jn+q_jN}f_j$ provided 
  $E(f_{j_0}|\cY)=0$. Rewrite
  \[
  \psi_N=\frac 1N\sum_{j=1}^N\big(\frac 1H\sum_{n=j}^{j+H-1}\sum_{i=1}^m
  T^{p_in+q_iN}f_i\big)+O(H/N)
  \]
  where $H$ will be chosen large but much smaller than $N$. By the convexity
  of the function $\vf(x)=x^2$ we have (up to $O(H/N)$),
  \[
  \psi_N^2\leq\frac 1N\sum_{j=1}^N\big(\frac 1H\prod_{n=j}^{j+H-1}\prod_{i=1}^m
  T^{p_in+q_iN}f_i\big)^2.
  \]
  By integration and the fact that $T$ is measure preserving,
  \begin{eqnarray*}
  &\|\psi_N\|^2_{L^2}\leq\frac 1N\sum_{j=1}^N\frac 1{H^2}\sum_{n,k=j}^{j+H-1}
  \int\prod_{i=1}^mT^{p_in+q_iN}f_iT^{p_ik+q_iN}f_id\mu\\
  &=\frac 1{NH^2}\sum_{j=1}^N\sum_{n,k=j}^{j+H-1}\int\prod_{i=1}^m
  T^{(p_i-p_1)n+(q_i-q_1)N}(f_iT^{p_i(k-n)}f_i)d\mu\\
  &=\frac 1{NH^2}\sum_{j=1}^N\sum_{n,k=j}^{j+H-1}\int\prod_{i=0}^{m-1}
  T^{\hat p_in+\hat q_iN}(f_{i+1}T^{p_{i+1}(k-n)}f_{i+1})d\mu
  \end{eqnarray*}
  where $\hat p_i=p_{i+1}-p_1$ and $\hat q_i=q_{i+1}-q_1$ satisfy conditions
  of Theorem \ref{thm2.2}. 
  
  Set $r=k-n$ and observe that a pair $(n,k)$ appears in the above sums only
  if $|r|=|k-n|<H$ and then for $H-|r|$ values of $j$ we rewrite the above
  estimate as
  \[
  \|\psi_N\|^2_{L^2}\leq\frac 1H\sum_{r=1-H}^{H-1}(1-\frac {|r|}H)\big(\frac 1N
  \sum_{n=1}^N\int\prod_{i=0}^{m-1}T^{\hat p_in+\hat q_iN}(f_{i+1}
  T^{p_{i+1}r}f_{i+1})d\mu\big)+O(\frac HN).
  \]
  Inserting conditional expectation inside the integral  and using $A_{m-1}$ 
  for a fixed $H$, every $r$ such that $|r|<H$ and $N$ large
  enough we can replace the integral term in the above inequality by
  \[
  \int\prod_{i=0}^{m-1}T^{\hat p_in+\hat q_iN}E(f_{i+1}T^{p_{i+1}r}f_{i+1}|
  \cY)d\mu.
  \]
  
  Hence, we obtain
  \begin{eqnarray}\label{3.4}
  &\|\psi_N\|^2_{L^2}\leq\frac 1H\sum_{r=1-H}^{H-1}(1-\frac {|r|}N)\\
  &\times\big(\frac 1N\sum_{n=1}^N\int\prod_{i=0}^{m-1}T^{\hat p_in+
  \hat q_iN}E(f_{i+1}T^{p_{i+1}r}f_{i+1}|\cY)d\mu\big)+O(H/N).\nonumber
  \end{eqnarray}
  Next, we estimate the integrals appearing in (\ref{3.4}) by
  \[
  \| E(f_{j_0}T^{p_{j_0}r}f_{j_0}|\cY)\|_{L^2}\prod_{j\ne j_0}\| f_j\|^2_\infty.
  \]
  Since $E(f_{j_0}|\cY)=0$ we obtain from $A_1$ for the case when $q_1=0$,
  which is proved as Lemma 8.1 in \cite{FKO} (where the ergodicity of 
  $\tilde T$ by the definition of weak mixing extensions is used), that
  \[
  \lim_{N\to\infty}\frac 1N\sum_{n=1}^N\int\big(E(f_{j_0}T^{p_{j_0}n}f_{j_0}
  |\cY)\big)^2d\nu=0.
  \]
  Hence, most of the terms in the right hand side of (\ref{3.4}) are small
  provided that $H$ is large enough. Since all terms in the right hand side 
  of (\ref{3.4}) are bounded by $\prod_{j=1}^m\| f_j\|^2_\infty$ and most of 
  them are small, their average in (\ref{3.4}) becomes arbitrarily small when
  $H$ and $N$ are large enough, completing the proof of Proposition 
  \ref{prop3.1}.
  \end{proof}
  
  Now Theorem \ref{thm2.1} is a particular case of (\ref{3.3}) considering a 
  trivial factor $\cY$, i.e. such that the corresponding $\sig$-algebra $\cB_1$
  contains only sets of zero or full measure. As to Theorem \ref{thm2.2} we
  will need the following corollary of Proposition \ref{prop3.1}.
  
  \begin{corollary}\label{cor3.2} Let $(X,\cB,\mu, T)$ be a relative weak 
  mixing extension of $(Y,\cD,\nu,S)$. If the action of $S$ on $\cD$ is GSZ,
  then so is the action of $T$ on $\cB$.
  \end{corollary}
  \begin{proof} The result follows immediately from (\ref{3.3}) in the same
  way as in Theorem 8.4 from \cite{FKO}.
  \end{proof}
  
  We observe that Proposition \ref{prop3.1} implies also that if 
  $(X,\cB,\mu, T)$ is a relative weak mixing extension of 
  $(X,\cB_1,\mu,T)$ and (\ref{2.3}) holds true for any $A\in\cB_1,\,\mu(A)>0$ 
  with $\liminf$ taken over all $N\to\infty$ then the same is true for any
  $A\in\cB,\,\mu(A)>0$, and so the restriction of $\liminf$ to $N\in\cN_A$
  comes not from relative weak extensions but from relative compact extensions 
  which will be studied below.
 
 \subsection{Relative compact extensions}\label{subsec3.3}
 For brevity and following \cite{FKO} we will drop here the word "relative"
 and will speak about compact extensions. Recall, that $(X,\cB,\mu,T)$ is
 said to be a compact extension of $(Y,\cD,\nu,S)$ if there exists a set
 $\cR\subset L^2(X,\cB,\mu)$ dense in $L^2(X,\cB,\mu)$ and such that for every
 $\del>0$ there exist functions $g_1,...,g_m\in L^2(X,\cB,\mu)$ satisfying
 \begin{equation}\label{3.5}
 \sup_{f\in\cR}\sup_{k\in\bbZ}\min_{1\leq j\leq m}\| T^kf-g_j\|_{L^2(\mu_y)}
 <\del\,\,\mbox{for $\nu$-almost all}\,\, y\in Y
 \end{equation}
 where, again, $\mu=\int\mu_yd\nu(y)$.
 
 As explained in Section \ref{subsec3.1} above the proof of Theorem 
 \ref{thm2.2} will be complete after we establish the following result.
 \begin{proposition}\label{prop3.3} Let $(X,\cB,\mu,T)$ be a compact extension 
 of $(Y,\cD,\nu,S)$. If the action of $S$ on $(Y,\cD,\nu)$ is GSZ then so is the
  action of $T$ on $(X,\cB,\mu)$.
  \end{proposition}
  
  \begin{proof} We will follow the proof of Theorem 9.1 from \cite{FKO}
  with a modification at the end.
  For an arbitrary $A\in\cB$ with $\mu(A)>0$ we have to show that 
  (\ref{2.3}) holds true. First, similarly to \cite{FKO} we conclude that
  without loss of generality the indicator function $f=\bbI_A$ of $A$ can
   be assumed to belong to the set $\cR$ appearing in the above definition of 
   compact extensions. We will assume for convenience that $T$ is ergodic,
   otherwise pass to an ergodic decomposition. Then $S$ is also ergodic.
   The condition $f\in\cR$ is equivalent to saying that the sequence 
   $\{ T^kf\}_{k\in\bbZ}$ is totally bounded, or relatively compact, in
   $L^2(\mu_y)$ for almost all $y$. Since $T\mu_y=\mu_{Sy}$ we conclude
   that the total boundedness of $\{ T^kf\}_{k\in\bbZ}$ in $L^2(\mu_y)$
   for $y$ in a set of positive measure already implies for an ergodic
   $S$ that $\{ T^kf\}_{k\in\bbZ}$ is totally bounded in a uniform manner
   in $L^2(\mu_y)$ for almost all $y$.
   
   Denote by $\oplus_{j=0}^\ell L^2(\mu_y)$ the direct sum of $\ell+1$ copies 
   of $L^2(\mu_y)$ endowed with the norm $\|(f_0,f_1,...,f_\ell)\|_y=\max_j
   \| f_j\|_{L^2(\mu_y)}$. It is clear that if $f\in\cR$ then the set
   \[
   \cL(\ell,f)=\{(f,T^{p_1n+q_1N}f,T^{p_2n+q_2N}f,...,T^{p_\ell n+q_\ell N}f)\}_
   {n,N\in\bbZ}
   \]
   is totally bounded in $\oplus_{j=0}^\ell L^2(\mu_y)$ for $\nu$-almost all
   $y\in Y$, in fact, uniformly in $y\in Y$. We write
   \[
   \cL(\ell,f,y)=\{(f,T^{p_1n+q_1N}f,T^{p_2n+q_2N}f,...,T^{p_\ell n+q_\ell N}
   f)_y\}_{n,N\in\bbZ}\subset\oplus_{j=0}^\ell L^2(\mu_y)
   \]
   where $(\cdot,...,\cdot)_y$ means that the vector function is considered on
    a fiber above $y\in Y$ and, recall, $f=\bbI_A\in\cR$. Throwing away
    $\nu$-measure zero set of $y$'s we can assume that uniform estimates hold
    true on the whole $Y$.
    
    Set $A_1=\{ y:\,\mu_y(A)>\mu(A)/2\}=\{y:\,\mu_y(A)>0\}$. Then $\nu(A_1)
    >\frac 12\mu(A)$. Indeed, this is clear if $\nu(A_1)=1$ while if  
    $\nu(A_1)<1$ then 
    \begin{eqnarray*}
    &\mu(A)=\int\mu_y(A)d\nu(y)=\int_{A_1}\mu_y(A)d\nu(y)+\int_{Y\setminus A_1}
    \mu_y(A)d\nu(y)\\
    &<\nu(A_1)+\frac 12\mu(A)(1-\nu(A_1)),
    \end{eqnarray*}
    and so $\frac 12\mu(A)<\nu(A_1)(1-\frac 12\mu(A))$.
    Thus, we can assume without loss of generality that
    $\mu_y(A)=0$ for all $y\not\in A_1$. We consider only $y\in A_1$ for
    which the corresponding elements of $\cL(\ell,f,y)$ have all nonzero 
    components, and so these elements have norm $\geq\sqrt {\frac 12\mu(A)}$
    in $L^2(\mu_y)$. The corresponding subset of $\cL(m,f,y)$ is denoted
    by $\cL^*(\ell,f,y)$ and it is still uniformly totally bounded. For each
    $y\in A_1$ and $\ve>0$ let $M(\ve,y)$ denote the maximum cardinality
    of $\ve$-separated sets in $\cL^*(\ell,f,y)$, which is a finite monotone
    decreasing piece-wise constant function of $\ve$ with at most countably
    many jumps. Since $M(\ve,y)$ is measurable as a function of $y$ there
    exist $\ve_0<\mu(A)/10\ell$, $\eta>0$ and $A_2\subset A_1$ with $\nu(A_2)
    >0$ so that $M(\ve,y)$ equals a constant $M$ for $\ve_0-\eta\leq\ve\ve_0$
    and $y\in A_2$.
    
    Take $y_0\in A_2$ and find integers $n_1,...,n_M$ and $N_1,...,N_M$ so that
    $\{(f,T^{p_1n_j+q_1N_j}f,...,T^{p_\ell n_j+q_\ell N_j}f)\},\, j=1,2,...,M,$
    is a maximal $\ve_0$-separated set in $\cL^*(\ell,f,y_0)$. Next, 
     $\| T^{p_ln_i+q_lN_i}f-T^{p_ln_j+q_lN_j}f\|_{L^2(\mu_y)}$, $1\leq i< 
     j\leq M$, $l=0,1,...,\ell,$ as functions on $Y$ are measurable and
     $y_0$ can be chosen so that each neighborhood of values of these functions
     at $y_0$ occurs with positive measure in the set $A_2$. Let now $A_3$ be
     the subset of $A_2$ of points $y$ such that 
     \begin{equation}\label{3.6}
     \| T^{p_ln_i+q_lN_i}f-T^{p_ln_j+q_lN_j}f\|_{L^2(\mu_y)}>
     \| T^{p_ln_i+q_lN_i}f-T^{p_ln_j+q_lN_j}f\|_{L^2(\mu_{y_0})}-\eta
     \end{equation}
     for any $i,j,l$ with $1\leq i\leq j\leq M$ and $0\leq l\leq \ell$. Then
     $\nu(A_3)>0$ by the choice of $y_0$.
     
     Now we use the assumption that the action of $S$ on $(Y,\cD,\nu)$ is 
     GSZ, applying it to $A_3$. Let $n,N\in\bbZ$, $n\leq N$ be such that
     \[
     \nu(\bigcap_{l=0}^\ell S^{-(p_ln+q_lN)}A_3)>0
     \]
     and let $y\in \bigcap_{l=0}^\ell S^{-(p_ln+q_lN)}A_3$. Since $S^{p_ln+q_lN}y
     \in A_3$ for $l=0,1,...,\ell,$ and $A_3\subset \bigcap_{l=0}^\ell 
     S^{-(p_ln_j+q_lN_j)}A_1$ for $j=1,...,M$ by the definition of $\cL^*(\ell,
     f,y)$ (together with (\ref{3.6}))
     then $S^{p_l(n_j+n)+q_l(N_j+N)}y\in A_1$ for $l=0,1,...,\ell$ and 
     $j=1,...,M$.
     
     Similarly to \cite{FKO} we conclude that the vectors $\{(f,T^{p_1(n+n_j)+
     q_1(N+N_j)}f,T^{p_2(n+n_j)+q_2(N+N_j)}f,...,T^{p_\ell(n+n_j)+
     q_\ell(N+N_j)}f),\, j=1,...,M,\}$ are $\ve_0-\eta$ separated in 
     $\cL^*(\ell,f,y)$ for $y\in\bigcap_{l=0}^\ell S^{-(p_ln+q_lN)}A_3$, and so 
     these vectors form a maximal such set which must be then $\ve_0-\eta$
     dense in $\cL^*(\ell,f,y)$. Since $(f,f,...,f)\in\cL^*(\ell,f,y)$ there
     exists $j$ such that $\{(f,T^{p_1(n+n_j)+q_1(N+N_j)}f,...,
     T^{p_\ell(n+n_j)+q_\ell(N+N_j)}f)\}$ is $\ve_0$-close to it. By the
     choice of $\ve_0$ this implies
     \begin{eqnarray*}
     &\mu_y\big(\bigcap_{l=0}^\ell T^{-(p_l(n+n_j)+q_l(N+N_j))}A\big)\\
     &=\int\prod_{l=0}^\ell T^{p_l(n+n_j)+q_l(N+N_j)}fd\mu_y\geq\frac 9{10}
     \mu_y(A)>\frac 13\mu(A).
     \end{eqnarray*}
     
     The index $j$ depends on $y$, so now we sum over $j$ to obtain that for
     each $y\in\bigcap_{l=0}^\ell S^{-(p_ln+q_lN)}A_3$,
     \[
     \sum_{j=1}^M\mu_y\big(\bigcap_{l=0}^\ell T^{-(p_l(n+n_j)+q_l(N+N_j)}A\big)
     \geq\frac 13\mu(A).
     \]
     Integrating over $\bigcap_{l=0}^\ell S^{-(p_ln+q_lN)}A_3$ we derive
     \[
     \sum_{j=1}^M\mu\big(\bigcap_{l=0}^\ell T^{-(p_l(n+n_j)+q_l(N+N_j)}A\big)
     \geq\frac {\mu(A)}3\nu(\bigcap_{l=0}^\ell S^{-(p_ln+q_lN)}A_3).
     \]
     Now we sum in $n$, $1\leq n\leq N$ and multiply by $\frac 1N$,
     \begin{eqnarray}\label{3.7}
    &\sum_{j=1}^M\frac 1N\sum_{1\leq n\leq N}\mu\big(\bigcap_{l=0}^\ell 
    T^{-(p_l(n+n_j)+q_l(N+N_j)}A\big)\\
     &\geq\frac {\mu(A)}3\frac 1N\sum_{1\leq n\leq N}\nu(\bigcap_{l=0}^\ell 
     S^{-(p_ln+q_lN)}A_3).\nonumber
     \end{eqnarray}
     
     Next, set $K_j(N)=N+N_j$. Then
     \begin{eqnarray}\label{3.8}
     &\big\vert\sum_{1\leq n\leq N}\mu\big(\bigcap_{l=0}^\ell 
    T^{-(p_l(n+n_j)+q_l(N+N_j)}A\big)\\
    &-\sum_{1\leq n\leq K_j(N)}\mu\big(\bigcap_{l=0}^\ell 
    T^{-(p_ln+q_lK_j(N))}A\big)\big\vert\leq 2|n_j|+|N_j|.\nonumber
    \end{eqnarray}
    
    Now we use the assumption that the action of $S$ on $(Y,\cD,\nu)$ is
    GSZ which implies that
    \begin{equation}\label{3.9}
    \liminf_{N\to\infty,\, N\in\cN}\frac 1N\sum_{1\leq n\leq N}
    \nu(\bigcap_{l=0}^\ell S^{-(p_ln+q_lN)}A_3)>0
    \end{equation}
    where $\cN$ is an infinite set of positive integers with bounded gaps.
    Define $\cN_j=\cN+N_j=\{ N+N_j:\, N\in\cN\}$, $j=1,...,M,$ which are also
    sets with bounded gaps.  Clearly,
    (\ref{3.9}) implies that there exists $\ve>0$ such that for any $N\in\cN$
    large enough
    \[
    \frac 1N\sum_{1\leq n\leq N}\nu(\bigcap_{l=0}^\ell S^{-(p_ln+q_lN)}A_3)>
    3\ve/\mu(A).
    \]
    Then by (\ref{3.7}) and (\ref{3.8}) we obtain that for any $N\in\cN$
    large enough
    \begin{equation}\label{3.10}
    \max_{1\leq j\leq M}\frac 1{K_j(N)}\sum_{1\leq n\leq K_j(N)}
    \mu(\bigcap_{l=0}^\ell T^{-(p_ln+q_lK_j(N))}A)\geq\frac \ve{2M}.
    \end{equation}
    
    Let 
    \[
    \cN_A=\{ N:\,\frac 1N\sum_{1\leq n\leq N}\mu(\bigcap_{l=0}^\ell 
    T^{-(p_ln+q_lN)}A)\geq\frac \ve{2M}\}.
    \]
    Then by (\ref{3.10}) for any $N\in\cN$ large enough there exists $j$ 
    such that $N+N_j\in\cN_A$. Hence, the gaps in $\cN_A$ are bounded by the
    bound on gaps of $\cN$ plus $2\max_{1\leq j\leq M}N_j$ and, clearly,
    \[
    \liminf_{N\to\infty,\, N\in\cN_A}\frac 1N\sum_{1\leq n\leq N}
    \mu(\bigcap_{l=0}^\ell T^{-(p_ln+q_lN)}A)\geq\frac \ve{2M}>0.
    \]
   This completes the proof of Proposition \ref{prop3.3}, as well, as of
   Theorem \ref{thm2.2}.
  \end{proof}

\section{Commuting transformations}\label{sec4}\setcounter{equation}{0}

In this section we will obtain Theorems \ref{thm2.4}, \ref{thm2.5} and
Corollary \ref{cor2.6}.

\subsection{Factors and extensions with respect to an abelian group of
transformations}\label{subsec4.1}
Let $G$ be a commutative group of transformations acting on $(X,\cB)$ so
that all $T\in G$ preserve a probability measure $\mu$ on $(X,\cB)$. A
probability space $(Y,\cD,\nu)$ is called a factor of $(X,\cB,\nu)$ if there 
exists an onto map $\pi:X\to Y$ such that $\pi\mu=\nu$ and $\pi^{-1}\cD=\cB$.
Define the action of $G$ on $(Y,\cD,\nu)$ by $T\pi x=\pi Tx$ for each
$T\in G$ and $x\in X$. This action preserves the
measure $\nu$ and we say that the system $(X,\cB,\mu,G)$ is an extension
of $(Y,\cD,\nu,G)$ and the latter is called a factor of the former. Clearly, 
this definition is compatible with the one given for one transformation in
Section 3.1.

Next, $(X,\cB,\mu,G)$ is called a relative weak mixing extension of
$(Y,\cD,\nu,G)$ if $(X,\cB,\mu,T)$ is a relative weak mixing extension of
$(Y,\cD,\nu,T)$ for each $T\in G,\, T\ne \mbox{id}$ as defined in Section 
\ref{subsec3.2}.
Furthermore, $(X,\cB,\mu,T)$ is called a (relative) compact extension of
$(Y,\cB,\nu,G)$ if (\ref{3.5}) holds true simultaneously for all $T\in G$
(with the same $\cR,\del$ and $g_1,...,g_m$) for $\nu$-almost all  $y\in Y$.
Finally, following \cite{Fu2} we call an extension $\al:\,(X,\cB,\mu,G)\to
(Y,\cD,\nu,G)$ primitive if $G$ is the direct product of two subgroups 
$G=G_c\times G_w$ where $(X,\cB,\mu,G_c)$ is a compact and $(X,\cB,\mu,G_w)$
 is a relative weak mixing extensions of $(Y,\cD,\nu,G_c)$ and of 
 $(Y,\cD,\nu,G_w)$, respectively.
 
 Next, $\cX=(X,\cB,\mu,G)$ as above will be called GSZ if (\ref{2.8}) 
 holds true for any $A\in\cB$ with $\mu(A)>0$ and all $T_j,\hat T_j\in G$, 
 $j=0,1,...,\ell,$ where the set $\cN_A$ depends on $A$ and $T_j,\hat T_j$'s,
 $T_0=\hat T_0=\mbox{id}$ and $T_1,...,T_\ell$ are distinct and different from the
 identity. Next, we rely on the Theorem 6.17 in \cite{Fu2} describing the 
 structure of extensions and 
 show similarly to Proposition 7.1 in \cite{Fu2} that if each $(X,\cB_\be,\mu,
 G)$ is GSZ for totally ordered (by inclusion) family of $\sig$-algebras then 
 $(X,\sup_\be\cB_\be,\mu,G)$ is also GSZ. It follows that in order to establish
 Theorem \ref{thm2.5} it suffices to show that any primitive extension
 $(X,\cB,\mu,G)$ of $(Y,\cD,\nu,G)$ is GSZ provided $(Y,\cD,\nu,G)$ is GSZ 
 itself.

\subsection{Weak mixing extensions}\label{subsec4.2}
The following result generalizes Proposition \ref{prop3.1} to the case
of several commuting transformations.
\begin{proposition}\label{prop4.1} Suppose that $(X,\cB,\mu, G)$ is a relative
 weak mixing extension of $(Y,\cD,\nu,G)$ where $G$ is a commutative group
of (both $\mu$ and $\nu$) measure preserving transformations as above. Let
$T_1,...,T_\ell\in G$ be distinct and different from identity while $\hat T_1,
...,\hat T_\ell$ be invertible (both $\mu$ and $\nu$) measure preserving
transformations of $(X,\cB,\mu)$ leaving $\cY=(Y,\cD,\nu)$ invariant and 
commuting with each other and with $T_1,...,T_\ell$. Then for each $m\geq 1$,
\begin{equation}\label{4.1}
\lim_{N\to\infty}\frac 1N\sum_{n=1}^N\int\big( E(\prod_{j=0}^mT^n_j\hat T^N_j
f_j|\cY)-\prod_{j=0}^mT^n_j\hat T^N_jE(f_j|\cY)\big)^2d\nu=0,
\end{equation}
where $T_0=\hat T_0=\mbox{id}$, and
\begin{equation}\label{4.2}
\lim_{N\to\infty}\|\frac 1N\sum_{n=1}^N\big(\prod_{j=1}^mT^n_j\hat T^N_jf_j
-\prod_{j=1}^mT^n_j\hat T^N_jE(f_j|\cY)\big)\|_{L^2}=0.
\end{equation}
\end{proposition}
\begin{proof} First, observe that considering a weak mixing extension of a
trivial factor we conclude that (\ref{4.2}) implies Theorem \ref{thm2.4}.
Denote the assertions (\ref{4.1}) and (\ref{4.2}) by $A_m$ and $B_m$, 
respectively, and prove them by induction showing that

(i) $A_{m-1}$ implies $B_m$ and

(ii) $B_m$ (for $\tilde X,\tilde\cB,\tilde\mu)$ and $\tilde T_j,
\tilde {\hat T}_j,\, j=1,...,\ell$) implies $A_m$ (for $(X,\cB,\mu)$ and
$T_j,\hat T_j,\, j=1,...,\ell$) where $\tilde X,\tilde\cB,\tilde\mu$ and
$\tilde T$ where defined in Section \ref{subsec3.2}.

First, observe that $A_0$ is obvious and $B_0$ does not play role here 
so we can denote by it any valid assertion. The proof proceeds essentially 
in the same way as for one transformation. We start with (ii) which is
easier. As in the one transformation case we assume first that $f_0$ is 
$\cB_1=\pi^{-1}(\cD)$-measurable. Then the integrals in (\ref{4.1}) have
the form
\begin{eqnarray*}
   &\int f^2_0\big(E(\prod_{j=1}^mT_j^n\hat T_j^N
  f_j|\cY)-\prod_{j=1}^mT_j^n\hat T_j^NE(f_j|\cY)\big)^2d\nu\\
  &\leq\sup|f^2_0|\int\big(E(\prod_{j=0}^{m-1}\bar T_j^n\bar {\hat T}_j^N
f_j|\cY)-\prod_{j=0}^{m-1}\bar T_j^n\bar {\hat T}_j^NE(f_j|\cY)\big)^2d\nu
  \end{eqnarray*}
  where $\bar T_j=T_{j+1}T_1^{-1}$ and $\bar {\hat T}_j=\hat T_{j+1}
  \hat T_1^{-1}$.
  Thus $A_m$ follows from $A_{m-1}$ if $f_0$ is $\cB_1$-measurable. Hence, as 
  in Section \ref{subsec3.2} we can assume that $E(f_0|\cY)=0$. Then the left 
  hand side of (\ref{4.1}) takes the form
\[
  \lim_{N\to\infty}\int f_0\otimes f_0\big(\frac 1N\sum_{n=1}^N\prod_{j=1}^m
  \tilde T_j^n\tilde {\hat T}_j^Nf_j\otimes f_j\big)d\tilde\mu.
  \]
  By $B_m$ for $(\tilde X,\tilde\cB,\tilde\mu)$ and $\tilde T_j,
  \tilde {\hat T}_j,\, j=1,...,m,$ the above limit equals
  \[
  \lim_{N\to\infty}\int f_0\otimes f_0\big(\frac 1N\sum_{n=1}^N\prod_{j=1}^m
  \tilde T_j^n\tilde {\hat T}_j^NE(f_j\otimes f_j|\cY)\big)d\tilde\mu.
  \]
  Since the sum here is $\cB_1$-measurable we can insert the conditional 
  expectation inside of the integral concluding as in Section \ref{subsec3.2}
  that the latter limit is zero completing the proof of (ii).
  
  In order to prove (i) we observe that
  \begin{eqnarray*}
  &\prod_{j=1}^mT_j^n\hat T_j^Nf_j-\prod_{j=1}^mT_j^n\hat T_j^NE(f_j|\cY)\\
  &=\sum_{j=1}^m(\prod_{i=1}^{j-1}T_i^n\hat T_i^Nf_i)T_j^n\hat T_j^N
  (f_j-E(f_j|\cY))\prod_{i=j+1}^mT_i^n\hat T_i^NE(f_i|\cY).
  \end{eqnarray*}
  This enables us to prove $B_m$ under the additional condition
  that for some $j_0,\, 1\leq j_0\leq m$ we have $E(f_{j_0}|\cY)=0$ (replacing
  $f_{j_0}$ by $f_{j_0}-E(f_{j_0}|\cY)$).
  
  It remains to show that $\lim_{N\to\infty}\|\psi_N\|_{L^2}=0$ for
  $\psi_N=\frac 1N\sum_{n=1}^N\prod_{j=1}^mT_j^^n\hat T_j^Nf_j$ provided 
  $E(f_{j_0}|\cY)=0$. Rewrite
  \[
  \psi_N=\frac 1N\sum_{j=1}^N\big(\frac 1H\sum_{n=j}^{j+H-1}\sum_{i=1}^m
  T_j^n\hat T_J^Nf_i\big)+O(H/N)
  \]
  where $H$ will be chosen large but much smaller than $N$. By  convexity
  of the function $\vf(x)=x^2$ we have (up to $O(H/N)$),
  \[
  \psi_N^2\leq\frac 1N\sum_{j=1}^N\big(\frac 1H\prod_{n=j}^{j+H-1}\prod_{i=1}^m
  T_i^n\hat T_i^Nf_i\big)^2.
  \]
  Integrating the above inequality we obtain
  \begin{eqnarray*}
  &\|\psi_N\|^2_{L^2}\leq\frac 1N\sum_{j=1}^N\frac 1{H^2}\sum_{n,k=j}^{j+H-1}
  \int\prod_{i=1}^mT_i^n\hat T_i^Nf_iT_i^k\hat T_i^Nf_id\mu\\
  &=\frac 1{NH^2}\sum_{j=1}^N\sum_{n,k=j}^{j+H-1}\int\prod_{i=0}^{m-1}
  \bar T_i^n\bar {\hat T}_i^N(f_{i+1}T_{i+1}^{(k-n)}f_{i+1})d\mu
  \end{eqnarray*}
  where $\bar T_i=T_{i+1}T_1^{-1}$, $\bar {\hat T}_i=\hat T_{i+1}\hat T_1^{-1}$
  and we observe that $\bar T_i,\, i=1,...,m-1,$ remain distinct and different
  from the identity.
  Writing $r=k-n$ we conclude similarly to Section \ref{subsec3.2} that
  this inequality implies that
  \begin{eqnarray}\label{4.3}
  &\|\psi_N\|^2_{L^2}\leq\frac 1H\sum_{r=1-H}^{H-1}(1-\frac {|r|}H)
  \big(\frac 1N\sum_{n=1}^N\int\prod_{i=0}^{m-1}\\
  &\bar T_i^n\bar {\hat T}_i^N(f_{i+1}T_{i+1}^rf_{i+1})d\mu\big)+O(H/N).
  \nonumber\end{eqnarray}
 Inserting conditional expectation inside the integral in the right hand side
 of (\ref{4.3}) and using $A_{m-1}$ 
  for a fixed $H$, every $r$ such that $|r|<H$ and $N$ large
  enough we can replace the integral term in the above inequality by
  \[
  \int\prod_{i=0}^{m-1}\bar T_i^n\bar {\hat T}_i^NE(f_{i+1}T_{i+1}^rf_{i+1}|
  \cY)d\mu
  \]
 which gives
 \begin{eqnarray}\label{4.4}
  &\|\psi_N\|^2_{L^2}\leq\frac 1H\sum_{r=1-H}^{H-1}(1-\frac {|r|}H)
  \big(\frac 1N\sum_{n=1}^N\int\prod_{i=0}^{m-1}\\
  &\bar T_i^n\bar {\hat T}_i^NE(f_{i+1}T_{i+1}^rf_{i+1}|\cY)d\mu\big)+O(H/N).
  \nonumber\end{eqnarray}
  
  Next, we estimate the integrals appearing in (\ref{4.4}) by
  \[
  \| E(f_{j_0}T_{j_0}^rf_{j_0}|\cY)\|_{L^2}\prod_{j\ne j_0}\| f_j\|^2_\infty.
  \]
  Since we assume that $E(f_{j_0}|\cY)=0$ then by $A_1$ for the case when
  $\hat T_1=\mbox{id}$ which is proved as Lemma 8.1 in \cite{FKO}
  (where ergodicity of $\tilde T_{j_0}$ is used which we know from the
  definition of relative weak mixing),
\[
\lim_{N\to\infty}\frac 1N\sum_{n=1}^N\int\big(Ef_{j_0}T_{j_0}^nf_{j_0}|\cY)
\big)^2=0.
\]
The concluding argument is the same as in Proposition \ref{prop3.1} which 
yields $A_m$ and completes the proof of Proposition \ref{prop4.1}.
\end{proof}

\subsection{Primitive extensions}\label{subsec4.3}
Let  $\al:\,\cX=(X,\cB,\mu,G)\to\cY=(Y,\cD,\nu,G)$ be a primitive extension, 
so that $G=G_c\times G_w$ with $\al:\,(X,\cB,\mu,G_c)\to(Y,\cD,\nu,G_c)$
 and $\al:\,(X,\cB,\mu,G_w)\to(Y,\cD,\nu,G_w)$ are relative compact and
 weak mixing extensions, respectively. Here $G$ is supposed to be a finitely
 generated free abelian group and $\mu=\int\mu_yd\nu(y)$. It follows from
 Proposition \ref{prop4.1} that
 \begin{lemma}\label{lem4.2}
 Let $S_1,...,S_m\in G_w$ be distinct and different from the identity,
 $\hat S_1,...,\hat S_m\in G$ be arbitrary and $f\in L^\infty(X)$. Define
 $\psi(y)=\int fd\mu_y$. Then for each $\ve,\del>0$ the number 
 $\#\cN_{\ve,\del,N}$ of elements of the set 
\begin{equation*}
\cN_{\ve,\del,N}=\{ n\leq N:\,\nu\{ y\in Y:\, |\int\prod_{i=1}^mS_i^n\hat 
S_i^Nfd\mu_y-\prod_{i=1}^m\psi(S^n_i\hat S_i^Ny)|>\ve\}>\del\}
\end{equation*}
satisfies
\begin{equation}\label{4.5}
\#\cN_{\ve,\del,N}\leq\gam_{\ve,\del}(N)N\quad\mbox{where}\quad 
\gam_{\ve,\del}(N)\to 0\,\,\mbox{as}\,\, N\to\infty
\end{equation}
denoting by $\#\Gam$ the cardinality of a set $\Gam$.
\end{lemma}
\begin{proof} Since $\psi=E(f|\cY)$ then by Proposition \ref{prop4.1},
\[
\lim_{N\to\infty}\frac 1N\sum_{n=1}^N\int\big( E(\prod_{i=1}^mS_i^n\hat S_i^N
|\cY)-\prod_{i=1}^mS_i^n\hat S_i^NE(f|\cY)\big)^2d\nu=0
\]
and (\ref{4.5}) follows.
\end{proof}

The implications of compactness which will be needed below are summarized in
the following lemma (see Lemma 7.10 in \cite{Fu2}).
\begin{lemma}\label{lem4.3} Let $A\in\cB$ with $\mu(A)>0$. Then we can find
a measurable set $A'\subset A$ with $\mu(A')$ as close to $\mu(A)$ as we like
and such that for any $\ve>0$ there exist a finite set of functions 
$g_1,...,g_K\in\cH=L^2(X,\cB,\mu)$ and a measurable function $k:\, Y\times G_c
\to\{1,...,K\}$ with the property that $\| R\bbI_{A'}-g_{k(y,R)}\|_y<\ve$
for $\nu$ almost all $y\in Y$ and every $R\in G_c$.
\end{lemma}

We will need also the following consequence of the multidimensional van der
Waerden theorem.
\begin{lemma}\label{lem4.4} (i) Let the number $K$ be given and let 
$T_1,T_2,...,T_H\in G$. There is a finite subset $\Psi\subset G$ and a number 
$M<\infty$ such that for any map $k:\, G\to\{ 1,2,...,K\}$ there exist
$T'\in \Psi$ and $m\in\bbN,\, 1\leq m\leq M$ such that
\[
k(T'T_i^m)=const,\, i=1,...,H;
\]
(ii) Let the number $K$ be given and $T_j,\hat T_j\in G,\, j=1,...,H$. There is
a finite set $\Psi\subset G$ and a number 
$M<\infty$ such that for any map $k:\, G\times G\to\{ 1,2,...,K\}$ satisfying
$k(T,S)=\hat k(TS)$ for some $\hat k:\, G\to\{ 1,2,...,K\}$ there exist
$T'\in \Psi$ and $m\in\bbN,\, 1\leq m\leq M$ such that
\[
k(T'T_i^m,\hat T_i^m)=const,\, i=1,...,H.
\]
\end{lemma}
\begin{proof} The assertion (i) is Lemma 7.11 in \cite{Fu2}. In order to
prove (ii) we apply (i) with $\hat k$ and $S_i=T_i\hat T_i,\, i=1,...,H,$
in place of $k$ and $T_1,...,T_H$, respectively, there. With $\Psi$ and $T'$
given by (i) for such $\hat k$ and $S_i$'s we obtain
\[
k(T'T_i^m,\hat T_i^m)=\hat k(T'(T_i\hat T_i)^m)=\hat k(T'S_i^m)=const
\]
for $i=1,...,H$.
\end{proof}

The following is the main result of this section which, as explained in 
Section \ref{subsec4.1}, yields Theorem \ref{thm2.5}.
\begin{proposition}\label{prop4.5}  Let $\al:\,\cX=(X,\cB,\mu,G)\to\cY=
(Y,\cD,\nu,G)$ be a primitive extension and $\cY$ be a GSZ system. Then
$\cX$ is also a GSZ system.
\end{proposition}
\begin{proof} We proceed similarly to Proposition 7.12 in \cite{Fu2} adapting
the proof there to our situation. Let $A\in\cB$ with $\mu(A)>0$ and let
$T_1,...,T_\ell,\hat T_1,...,\hat T_\ell\in G$. Replacing $A$ by a slightly
smaller set, we can assume that $\bbI_A$ has the compactness property described
 in Lemma \ref{lem4.3}. Writing $\mu(A)=\int\mu_y(A)d\nu(y)$, we see that
 there exists a measurable subset $B\subset Y,\,\nu(B)>0$ with $\mu_y(A)>a=
 \mu(A)/2$ for all $y\in B$. We express $T_j,\hat T_j$ as products of elements
 in $G_c$ and in $G_w$ and assume without loss of generality that for all
 $n\leq N$,
 \begin{equation}\label{4.6}
 \{ T^n_j\hat T^N_j,\, j=1,...,\ell,\}\subset\{ R^n_iS^n_j\hat R_i^N\hat S_j^N,
 \, i=1,...,r;\, j=1,2,...,s,\}
 \end{equation}
 where $R_1=\hat R_1=\mbox{id}$, $R_i,\hat R_i\in G_c,\, i=1,...,r$, 
 $S_j,\hat S_j\in G_w,\, j=1,...,s,$ and $S_1,...,S_s$ are distinct. Since the
 set of transformations in the right hand side of (\ref{4.6}) is at least as
 large as the one in the left hand side of (\ref{4.6}) then (\ref{2.8})
 will follow if we prove that for an infinite syndetic set $\cN_A\subset\bbN$,
 \begin{equation}\label{4.7}
 \lim_{N\to\infty,\, N\in\cN_A}\frac 1N\sum_{n=1}^N\mu\big(\bigcap_{i,j}
 (R^n_iS^n_j\hat R_i^N\hat S_j^N)^{-1}A\big)>0.
 \end{equation}
 
 Let $a_1<a^s$. We will show that there exist an infinite syndetic set
 $\cN_A\subset\bbN$ and $\ve>0$ such that for each $N\in\cN_A$ there exist a 
 subset $P_N\subset\{ 1,2,...,N\}$ with $\# P_N\geq\ve N$ and $\eta>0$ such
 that for every $n\in P_N$ we can find a set $B_{n,N}\subset Y$, $B_{n,N}\in
 \cD$ with $\nu(B_{n,N})>\eta$ satisfying
 \begin{equation}\label{4.8}
 \mu_y\big(\bigcap_{i,j}(R^n_iS^n_j\hat R_i^N\hat S_j^N)^{-1}A\big)>a_1\,\,\,
 \mbox{for all}\,\,\, y\in B_{n,N}.
 \end{equation}
 Integrating the inequality (\ref{4.8}) over $B_{n,N}$ and taking into 
 account (\ref{4.6}) we obtain that for any $N\in\cN_A$,
 \begin{eqnarray*}
 &\frac 1N\sum_{n=1}^N\mu\big(\bigcap_{j=0}^\ell(T_j^n\hat T_j^N)^{-1}A\big)\\
 &\geq\frac 1N\sum_{n=1}^N\mu\big(\bigcap_{i,j}(R^n_iS^n_j\hat R_i^N
 \hat S_j^N)^{-1}A\big)>\ve a_1\eta
 \end{eqnarray*}
 and both (\ref{4.7}) and (\ref{2.8}) will follow.
 
 The set $B_{n,N}$ will be determined by two requirements. For $a_1<a_2<a^s$
 we will require that
 \begin{equation}\label{4.9}
 \mu_y(\bigcap_jS_j^{-n}\hat S_j^{-N}A)>a_2
 \end{equation}
 whenever $n\in P_N$ and $y\in B_{n,N}$. Choose $\ve_1>0$ such that if
 \begin{equation}\label{4.10}
 \mu_y(S_j^{-n}R_i^{-n}\hat S_j^{-N}\hat R_i^{-N}A\triangle S_j^{-n}
 \hat S_j^{-N}A)<\ve_1\,\,\mbox{for all}\,\, 1\leq i\leq r,\, 1\leq j\leq s
 \end{equation}
 (where $\triangle$ denotes the symmetric difference) then (\ref{4.9}) 
 implies (\ref{4.8}). Then we require that (\ref{4.10}) holds true for any
 $n\in P_N$ and $y\in B_{n,N}$.
 
 Suppose now that $P_N$ and $\{ B_{n,N},\, n\in P_N,\, N\in\cN_A\}$ have
 been found so that (\ref{4.10}) is satisfied for all $n\in P_N$,
  $y\in B_{n,N}$ and, in addition, 
  \begin{equation}\label{4.11}
  S_j^n\hat S_j^Ny\in B\quad\mbox{for all}\quad y\in B_{n,N},\, 1\leq n\leq N,
  \, 1\leq j\leq s.
  \end{equation}
  Now, applying Lemma \ref{lem4.2} with $f=\bbI_A,\,\ve<a^s-a_2$ and $\del<
  \frac 12\eta$ we obtain
  \[
  \mu_y(\bigcap_iS_i^{-n}S_j^{-N}A)=\int\prod_jS_j^n\hat S_j^Nfd\mu_y>\prod_j
  \psi(S^n_j\hat S_j^Ny)-\ve\geq a^s-\ve>a_2,
  \]
  with $\psi$ defined in Lemma \ref{lem4.2},
  for all $y\in B_{n,N}$ except for a set $\hat B_{n,N}$ of $y$'s of measure
  $\nu$ less than $\frac 12\eta$ and for $n\not\in\cN_{\ve,\del,N}$. Set
  $\tilde P_N=P_N\setminus\cN_{\ve,\del,N}$ and $\tilde B_{n,N}=
  B_{n,N}\setminus\hat B_{n,N}$ then considering new $P_n=\tilde P_N$ and 
  $B_{n,N}=\tilde B_{n,N}$ we obtain (\ref{4.9}). The problem is reduced to
  finding $P_N$ and $B_{n,N}$ such that (\ref{4.10}) and (\ref{4.11}) are
  satisfied.
  
  Next, we replace (\ref{4.10}) by the requirement that there exists $g\in
  \cH_y=L^2(X,\cB,\mu_y)$ such that
  \begin{equation}\label{4.12}
  \| S_j^nR^n_i\hat S_j^N\hat R^N_i-S_j^n\hat S_j^Ng\|_y<\ve_2,\, 1\leq i
  \leq r,\, 1\leq j\leq s
  \end{equation}
  (where $\|\cdot\|_y=\|\cdot\|_{L^2(X,\mu_y)}$) with $\ve_2<\frac 12\sqrt 
  {\ve_1}$. Since $R_1=\hat R_1=\mbox{id}$ we will have 
  \[
  \| S_j^nR_i^n\hat S_j^N\hat R_i^N\bbI_A-S^n_j\bbI_A\|_y<2\ve_2<\sqrt 
  {\ve_1}
  \]
  which gives (\ref{4.10}) since
  \[
  \| T\bbI_A-S\bbI_A\|^2_y=\int |\bbI_{T^{-1}A}(x)-
  \bbI_{S^{-1}A}(x)|^2d\mu_y(x)=\mu_y(T^{-1}A\triangle S^{-1}A).
  \]
  
 Now recall that $A$ was chosen to comply with conditions of Lemma 
 \ref{lem4.3}. We can therefore find $g_1,g_2,...,g_K\in L^2(X,\mu)$
 and a function $k:\, Y\times G_c\to\{ 1,2,...,K\}$ so that 
 $\| R\bbI_A-g_{k(y,R)}\|_y<\ve_2$ for every $R\in G_c$ and $\nu$-almost all 
 $y$. We define now a sequence of functions $k_{q,Q}:\, Y\times G\times G\to
 \{ 1,2,...,K\}$ by 
 \[
 k_{q,Q}(y,SR,\hat S\hat R)=k(S^q\hat S^Qy,R^q\hat R^Q)
 \]
 for integers $1\leq q\leq Q$ and transformations $R,\hat R\in G_c,\, S,\hat S
 \in G_w$. This is well defined since $G=G_c\times G_w$ is a direct product.
 Then for $\nu$-almost all $y$,
 \begin{equation}\label{4.13}
 \| S^qR^q\hat S^Q\hat R^Q\bbI_A-S^q\hat S^Qg_{k_{q,Q}(y,RS,\hat R\hat S)}\|_y
 =\| R^n\hat R^N\bbI_A-g_{k(S^q\hat S^Qy,R^q\hat R^Q)}\|_{S^q\hat S^Qy}<\ve_2.
 \end{equation}
 
 Fix $q\leq Q$ and $y$ for which (\ref{4.13}) holds true and apply Lemma 
 \ref{lem4.4}(ii) to the function $k(\cdot,\cdot)=k_{q,Q}(y,\cdot,\cdot)$
 on $G\times G$. Independently of $q, Q$ and $y$ there is a finite set
 $\Psi\subset G$ and a number $M$ such that $k_{q,Q}(y,T'R_i^mS_j^m,
 \hat R^m_i\hat S^m_i)$ takes on the same value $k$ for $1\leq i\leq r$, 
 $1\leq j\leq s$, for some $T'\in\Psi$ and some $m$ with $1\leq m\leq M$.
 Then if $T'=R'S'$ and $g_{(q,y)}$ is the corresponding $g_k$ we obtain
 from (\ref{4.13}) for $1\leq i\leq r$, $1\leq j\leq s$ that
 \begin{eqnarray}\label{4.14}
 &\| S_j^{qm}R_i^{qm}\hat S_j^{Qm}\hat R_i^{Qm}\bbI_A-S_j^{qm}\hat S_j^{Qm}
 ((R')^{-q}g_{(q,y)})\|_{(T')^qy}\\
 &=\|(T')^qS_j^{qm}R_i^{qm}\hat S_j^{Qm}\hat R_i^{Qm}\bbI_A-(T')^q(R')^{-q}
 S_j^{qm}\hat S_j^{Qm}g_{(q,y)}\|_{y}\nonumber\\
 &=\|(S'S_j^m)^q(R'R_i^m)^q\hat S_j^{Qm}\hat R_i^{Qm}\bbI_A-(S'S_j^m)^q\hat 
 S_j^{Qm}g_{(q,y)}\|_y<\ve_2\nonumber
 \end{eqnarray}
 where we took into account that $g_{(q,y)}=g_k=g_{k_{q,Q}(y,(R'R_i^m)(S'S_j^m),
 \hat R^m_i\hat S_i^m)}$. We have shown that for every $q=1,...,Q$, $Q\in
 \bbN$ and $\nu$-almost all $y\in Y$ there exist $m$ and $T'$, both having
  a finite range of possibilities, such that (\ref{4.12}) is satisfied with
  $n=qm$ and $N=Qm$ for $(T')^qy$ in place of $y$.
  
  Next, we will produce the set $P_N$ and the sets $B_{n,N},\, n\in P_N$ such
  that both (\ref{4.11}) and (\ref{4.12}) are satisfied for $(y,n),\,
  y\in B_{n,N}$. For each $q$ form the set
  \[
  C_q=\bigcap_{j,m,T'}S_j^{-mq}\hat S_j^{-Qm}(T')^{-q}B\subset Y
  \]
  where the intersection is taken over $j,m,T'$ with $1\leq j\leq s,\, 
  1\leq m\leq M,\, T'\in\Psi$. Using the fact that $(Y,\cD,\nu)$ is 
  a GSZ system we conclude that for each $Q$ from an infinite syndetic set 
  $\cN'\subset\bbN$ there exists $P'_Q\subset\{ 1,...,Q\}$ with $\# P'_N
  \geq\ve Q$ for some $\ve>0$ independent of $Q$ and such that $\nu(C_q)>
  \eta'$ for some $\eta'>0$ and all $q\in P'_Q$.
  
 Now let $y\in C_q$ for $q\in P'_Q$. There exist $m=m(q,y)$ and $T'=T'(q,y)$
 such that $(T')^qy$ (in place of $y$) satisfies (\ref{4.12}) for $n=qm,\,
 N=Qm$ and $q\leq Q$. In addition, $(T')^qy$ also satisfies (\ref{4.11})
 for these $T'$ and $m$ taking into account that
  by the definition of $C_q$ this condition is
 satisfied with $n=mq,\, N=mQ$ by all $(T')^qy$ and all $m$ such that 
 $T'\in\Psi,\, 1\leq m\leq M$ since $(T')^qy\in\bigcap_{j,m}S_j^{-mq}
 \hat S_j^{-mQ}B$ whenever $y\in C_q$, and so $S_j^{mq}\hat S_j^{mQ}(T')^qy
 \in B$.
 
 Let $J$ be the total number of possibilities for $(m,T')$. Then for a subset
 $D_q\subset C_q$ with $\nu(D_q)>\frac {\eta'}{J}$, $m(q,y)$ and $T'(q,y)$
 take a constant value, say, $m(q)$ and $T'(q)$, respectively. We now define
 $n(q)=qm(q)$ and set $P_Q=\{ n(q)\leq Q:\, q\in P'_Q\}$ and $B_{n(q)}=
 (T'(q))^qD_q$. Then $\nu(B_{n(q)})=\nu(D_q)>\eta'/J$, $S_j^{n(q)}
 \hat S_j^{m(q)Q}B_{n(q)}\in B,\, j=1,...,s,$ and
 \[
 \| S_j^{n(q)}\hat S_j^{m(q)Q}R_i^{n(q)}\hat R_i^{m(q)Q}\bbI_A-S_j^{n(q)}
 \hat S_j^{m(q)Q}g'_{(q,y)}\|_y<\ve_2
 \]
 for $y\in B_{n(q)},\, 1\leq i\leq r,\, 1\leq j\leq s$ for an appropriately
 defined $g'_{(g,y)}$. Finally, $\#\{ n(q),\, q\leq Q\}\geq\ve/M$ and the
 gaps of the set $\{ m(q)Q,\, Q\in\cN'\}$ are bounded by $M$ times of the
 maximal gap of $\cN'$. This complets the proof of Proposition \ref{prop4.5},
 as well as of Theorem \ref{thm2.5}.
\end{proof}

\section{Short proofs of Theorems \ref{thm2.2} and \ref{thm2.5}}\label{sec5}
\setcounter{equation}{0}

Recall that $F_k\subset\bbZ^d,\, k=1,2,...$ is called a F\o lner sequence if
the cardinality of the symmetric difference $(\bar n+F_k)\triangle F_k$ is
o$(|F_k|)$ as $k\to\infty$ for any $\bar n\in\bbZ^d$.
Now, suppose that for any F\o lner sequence $F_k\subset\bbZ^2$, $k=1,2,...$,
\[
\liminf_{k\to\infty}\frac 1{|F_k|}\sum_{(n,m)\in F_k}a_{n,m}>0 
\]
(in fact, we will need this only when $F_k$'s are squares). Then there
exists $\ve>0$ and an integer $M\geq 1$ such that in any square 
$R\subset\bbZ^2$ with
 the side of length $M$ we can find $(n,m)\in R$ such that $a_{n,m}>\ve$.
 Indeed, if this were not true then we could find a sequence of squares
 $R_j\subset\bbZ^2$ with sides of length $M_j\to\infty$ 
 as $j\to\infty$ and a sequence
 $\ve_j\to 0$ as $j\to\infty$ such that $a_{n,m}\leq\ve_j$ for all $(n,m)
 \in R_j$. Then, of course,
 \[
 \liminf_{j\to\infty}\frac 1{|R_j|}\sum_{(n,m)\in R_j}a_{n,m}\leq\liminf_{j\to
 \infty}\ve_j=0,
 \]
 which contradicts our assumption since $\{ R_j\}_{j=1}^\infty$ is a F\o lner
 sequence. Clearly, this argument remains true for any $\bbZ^d$ replacing
 squares by $d$-dimensional boxes but we will not need this here.
 
 Now, let $M,\ve>0$ be numbers whose existence was established above and
 assume that $a_{n,m}\geq 0$ for all integer $n$ and $m$. Set
 $Q_j=\{(n,m):\, j(M+1)\leq m<(j+1)(M+1)$ and $0<n\leq j(M+1)\}$. Then $Q_j$ 
 contains $j$ disjoint squares with the side of length $M$, and so
 \[
 \sum_{(n,m)\in Q_j}a_{n,m}\geq\ve j.
 \]
 Hence, there exists $j(M+1)\leq N_j<(j+1)(M+1)$ such that
 \[
 \sum_{n=1}^{N_j}a_{n,N_j}\geq\frac {\ve j}{M+1}.
 \]
 Clearly, $\cN=\{ N_j,\, j=1,2,...\}$ is a set of integers with gaps bounded 
 by $2M$ and
 \[
 \liminf_{j\to\infty}\frac 1{N_j}\sum_{n=1}^{N_j}a_{n,N_j}\geq\frac \ve{(M+1)^2}.
 \]
 
 Next, we will apply the above arguments to the situation of Theorem 
 \ref{thm2.5}. Let $T_j,\hat T_j,\, j=1,...,\ell,$ be as in Theorem \ref{thm2.5}
 commuting measure preserving transformations of a measure space $(X,\cB,\mu)$
 and set $S_j^{(n,m)}=(T_{j-1}^n\hat T_{j-1}^m)^{-1}T_j^n\hat T_j^m,\,
 j=1,...,\ell,$ with
 $S_0^{(n,m)}$ being the identity transformation. Then $S_j^{(n,m)},\, 
 j=0,1,...,\ell,$ are commuting measure preserving transformations of
 $(X,\cB,\mu)$ and $T_j^n\hat T_j^m=S_0^{(n,m)}S_1^{(n,m)}\cdots S_j^{(n,m)},
 \, j=0,1,...,\ell$. Now, it follows from Theorem B of \cite{Au} that for
 any set $A\in\cB$ with $\mu(A)>0$ and any F\o lner sequence $F_k\subset\bbZ^2$,
 \begin{eqnarray*}
 &\lim_{k\to\infty}\frac 1{|F_k|}\sum_{(n,m)\in F_k}\mu\big(\bigcap_{j=0}^\ell
 (T_j^n\hat T_j^m)^{-1}A\big)\\
 &=\lim_{k\to\infty}\frac 1{|F_k|}\sum_{(n,m)\in F_k}\mu\big(\bigcap_{j=0}^\ell
 (S_0^{(n,m)}S_1^{(n,m)}\cdots S_j^{(n,m)})^{-1}A\big)>0,
 \end{eqnarray*}
 i.e. the limit exists and it is positive. Taking 
 $a_{n,m}=\mu\big(\bigcap_{j=0}^\ell(T_j^n\hat T_j^m)^{-1}A\big)$ we obtain by 
 the above arguments that there
 exists an infinite set with bounded gaps $\cN_A$ such that (\ref{2.8})
 holds true, completing the proof of Theorem \ref{thm2.5}.   \qed
 
 Next, we derive a polynomial version of Theorem \ref{thm2.2}. Replace in 
 (\ref{2.1}) the linear terms $p_in+q_iN$ by general polynomials $p_i(n,N),
 \, i=1,...,\ell,$ taking on integer values on integer pairs $n,N$ and such
 that for each $k\in N$ there exists a pair $n,N$ with $p_i(n,N),\, i=1,...,
 \ell,$ all divisible by $k$. Then by Theorem 1.4 in \cite{BLL},
 \[
 \lim_{k\to\infty}\frac 1{|F_k|}\sum_{(n,m)\in F_k}\mu(\bigcap_{i=0}^\ell
 T^{-p_i(n,m)}A)>0
 \]
 for every $A$ with $\mu(A)>0$ and any F\o lner sequence $F_k\subset\bbZ^2$.
 Set $a_{n,m}=\mu(\bigcap_{i=0}^\ell T^{-p_i(n,m)}A)$. Then by the above 
 argument
 there exists an infinite set of positive integers $\cN$ with uniformly bounded
  gaps such that
  \begin{equation*}
  \liminf_{N\to\infty,\, N\in\cN}\frac 1N\sum_{n=1}^N\mu(\bigcap_{i=0}^\ell 
  T^{-p_i(n,N)}A)>0
  \end{equation*}
 providing a polynomial version of (\ref{2.3}).  \qed

\section{Nonconventional polynomial arrays}\label{sec6}
\setcounter{equation}{0} 
\subsection{Proof of Theorem \ref{thm2.8}}\label{subsec6.1}
We start with the proof of Theorem \ref{thm2.8} which proceeds close
to the proof of Theorem D in \cite{BL}. First, by changing functions $f_j$
we can always assume without loss of generality that $P_{ij}(0)=0,\, Q_{ij}(0)
=0,\, i=1,...,\ell,\, j=1,...,k$. If $\ell=1$ and $P_{11}(n,N)=pn$
where $p$ is an integer and $P_{1j}(n)\equiv 0$ when $j>1$ while $Q_{1j}(N)$'s
are functions of $N$ taking on integer values on integers then for any 
measurable $L^2$ function $f$,
\begin{eqnarray}\label{6.1}
&\int(\frac 1N\sum_{n=1}^NT_1^{pn}\hat T_1^{Q_{11}(N)}\cdots
\hat T^{Q_{1k}(N)}f-\int fd\mu)^2d\mu\\
&=\int(\frac 1N\sum_{n=1}^NT^{pn}f-\int fd\mu)^2d\mu\to 0\quad
\mbox{as}\,\, N\to\infty\nonumber
\end{eqnarray}
since $T$ is weakly mixing, and so $T^{p}$ is weakly mixing and,
 in particular, ergodic, and so the result follows from the $L^2$ ergodic
 theorem.

In order to deal with the general case of Theorem \ref{thm2.8} we will
 need the following version of the van der Corput theorem whose proof
 is the same as of Theorem 1.4 in \cite{Be} (see also Theorem 1.5 there),
  and so we refer the reader there. This follows also from uniform
  versions of the van der Corput theorem (see, for instance, \cite{Le}).
  
\begin{lemma}\label{lem6.1} Let $\{ x_{n,N}\}_{n=1}^N,\,N=1,2,...$ be
 a bounded sequence of vectors in a Hilbert space such that
 \begin{equation}\label{6.2}
D-\lim_h\lim_{N\to\infty}\frac 1N\sum_{n=1}^N\langle x_{n,N},\, 
 x_{n+h,N}\rangle=0
 \end{equation}
 where $\langle\cdot,\cdot\rangle$ is the inner product and $D-\lim_h$ denotes 
 the limit as $h\to\infty$ outside a set of integers having zero upper density.
  Then
 \begin{equation}\label{6.3}
 \lim_{N\to\infty}\|\frac 1N\sum_{n=1}^Nx_{n,N}\|=0
 \end{equation}
 where $\|\cdot\|$ is the Hilbert space norm.
 \end{lemma}

Next, we will describe the "PET induction" in our circumstances where we
closely follow \cite{BL} and refer the reader there for more details.
Let $P_j,\, j=1,...,k,$ be any polynomials and $Q_j,\, j=1,...,k$ be any
functions taking on integer values on integers and such that $P_j(0)=Q_j(0)=0$.
Similarly to \cite{BL} we will call
\[
\vf(n)=T_1^{P_1(n)}\cdots T_k^{P_k(n)},\, \psi(n)=\hat T_1^{Q_1(n)}\cdots 
\hat T_k^{Q_k(n)}\,\,\mbox{and}\,\,\Phi(n,N)=\vf(n)\psi(N)
\]
$P$-polynomial expressions where $P$ indicates the fact that $Q_i$'s are not
necessarily polynomials. Products of $P$-polynomial expressions and their
inverses are $P$-polynomial expressions, and so they form a group $PE$. Clearly,
if $\Phi(n,N)=\vf(n)\psi(N)\in PE$ then
$\Phi^{-1}(n_0,N)\Phi(n+n_0,N)=\vf^{-1}(n_0)\vf(n+n_0)\in PE$. The degree,
deg$(\vf(n))$ of $\vf(n)=T_1^{P_1(n)}\cdots T_k^{P_k(n)}$ is the maximal
degree of polynomials $P_j,\, j=1,...,k$ and the degree, deg$(\Phi(n,N))$
of a $P$-polynomial expression $\Phi(n,N)=\vf(n)\psi(N)$ is defined as the 
degree of $\vf$. Again, following \cite{BL} we define the weight of a 
$P$-polynomial expression $\Phi(n,N)=\vf(n)\psi(N)$ with $\vf(n)=
T_1^{P_1(n)}\cdots T_k^{P_k(n)}$ as the pair $(r,d)$ such that deg$P_{r+1}=...
=$deg$P_k(n)=0$, deg$P_r(n)=d\geq 1$. The weight $(r,c)$ is greater than 
$(s,d)$ if $r>s$ or if $r=s$ and $c>d$.

Two $P$-polynomial expressions $\Phi_1(n,N)=\vf_1(n)\psi_1(N)$ and
$\Phi_2(n,N)=\vf_2(n)\psi_2(N)$ with $\vf_1(n)=T_1^{P^{(1)}_1(n)}\cdots 
T_k^{P^{(1)}_k(n)}$ and $\vf_2(n)=T_1^{P^{(2)}_1(n)}\cdots T_k^{P^{(2)}_k(n)}$
are called equivalent if they have the same weight $(r,d)$ and the leading
coefficient of the polynomials $P_r^{(1)}$ and $P_r^{(2)}$ coincide, as well.
Any finite subset of $PE$ is called a system and the degree of a system is the
maximal degree of its elements. To every system a weight matrix $(N_{rd},\,
1\leq r\leq k,\, 1\leq d\leq D)$ is associated where $N_{rd}$ is the number of
equivalence classes formed by the elements of the system whose weights are
$(r,d)$ and $D$ is the maximal degree of the polynomials $P_{ij}$ appearing in 
Theorem \ref{thm2.8}. As in \cite{BL} we say that the weight matrix $M'=
(N'_{rd},\, 1\leq r\leq k,\, 1\leq d\leq D)$ precedes the weight matrix
$M=(N_{r,d},\, 1\leq r\leq k,\, 1\leq d\leq D)$ if for some $(r_0,d_0)$, 
$N'_{r_0d_0}=N_{r_0d_0}-1,
\, N'_{rd}=N_{rd}$ when $r\geq r_0$ and $d\geq d_0$ except for $r=r_0$ and
$d=d_0$, $N_{rd}=0$ and $N'_{rd}$ are arbitrary nonnegative integers when
$r\leq r_0$ and $d\leq d_0$ except for $r=r_0$ and $d=d_0$ (for a picture
explanation see \cite{BL}). 

Now observe that the system appearing in 
(\ref{6.1}) has the weight matrix $M_0=(N_{rd})$ where $N_{11}=1$ and 
$N_{rd}=0$ if $(r,d)\ne (1,1)$. Thus, (\ref{6.1}) proves Theorem \ref{thm2.8}
for any system with the weight matrix $M_0$. Next, we proceed step by step 
considering systems with weight matrices $M_0,M_1,M_2,...,M_K$ such that each
$M_i$ preceeds $M_{i+1},\, i=0,1,...,K-1$ arriving finally to the matrix $M_K$
with arbitrary predefined weights $N_{rd},\, 1\leq r\leq k,\, 1\leq d\leq D$
(for a graphical explanation of this see \cite{BL}). Our goal is to show that
if Theorem \ref{thm2.8} is valid for any
system with the weight matrix $M_i$ then it is valid for any system with the
weight matrix $M_{i+1}$ which by induction will yield Theorem \ref{thm2.8}.

Next, we remark that without loss of generality we can assume that 
$\int f_id\mu=0$ for any $i=1,...,\ell$ which is the result of the equality
\[
\prod_{i=1}^\ell a_i-\prod_{i=1}^\ell b_i=\sum_{E\subset\{ 1,...,\ell\},E\ne
\emptyset}\prod_{i\in E}(a_i-b_i)\prod_{i\not\in E}b_i.
\]
Indeed, taking $a_i=T_1^{P_{i1}(n)}\cdots T_k^{P_{ik}(n)}\hat T_1^{Q_{i1}(N)}
\cdots\hat T_k^{Q_{ik}(N)}f_i$ and $b_i=\int f_id\mu$ we transform the left
hand side of (\ref{2.11}) into a sum of similar product expressions where
all functions have zero integrals and the result to be proved now is that all
corresponding limits are zero. Thus, writing
\[
\Phi_i(n,N)=\vf_i(n)\psi_i(N)\,\,\mbox{with}\,\,\vf_i(n)=T_1^{P_{i1}(n)}\cdots
T_k^{P_{ik}(n)}
\]
\[
\mbox{and}\,\,\psi_i(N)=\hat T_1^{Q_{i1}(N)}\cdots\hat T_k^{Q_{ik}(N)}
\]
we have to prove that
\begin{equation}\label{6.4}
\lim_{N\to\infty}\bigg\|\frac 1N\sum_{n=0}^{N-1}\prod_{i=1}^\ell\Phi_i(n,N)
f_i\bigg\|_{L^2}=0.
\end{equation}

As in \cite{BL} we can assume without loss of generality that $T_1,...,T_k$
are linearly independent elements of the basis of the finitely generated free
abelian group $G$. Then $\vf(n)=T_1^{P_1(n)}\cdots T_k^{P_k}=\mbox{id}$ for some
polynomials $P_1,...,P_k$ implies $P_1=\cdots =P_k=0$. By Lemma \ref{lem6.1},
(\ref{6.4}) would follow if
\begin{equation}\label{6.5}
D-\lim_h\lim_{N\to\infty}\frac 1N\sum_{n=0}^{N-1}L(n,h)=0
\end{equation}
where
\[
L(n,h)=\big\langle\prod_{i=1}^\ell\Phi_i(n,N)f_i,\,\prod_{i=1}^\ell\Phi_i
(n+h,N)f_i\big\rangle .
\]

 Next, we will need the following result.
 
 \begin{lemma}\label{lem6.2} Let nonconstant polynomials $P_1(n,N),\, 
 P_2(n,N),\, ..., P_k(n,N)$ of $n$ and $N$ be essentially distinct and
 nontrivially depend on $n$. Then
 for each sufficiently large $h$ the polynomials $P_1(n,N),\, P_2(n,N),\,
  ...,P_k(n,N),\, P_1(n+h,N),\,...,P_k(n+h,N)$ are pairwise essentially 
  distinct (where $h$ is viewed as a  constant) except for pairs
  $P_i(n,N),\, P_i(n+h,N)$ where $P_i(n,N)=p_in+Q_i(N)$ and then
  $P_i(n+h,N)-P_i(n,N)=a_ih$.
 \end{lemma}
 \begin{proof} Clearly, $P_1(n+h,N),\,..., P_k(n+h,N)$ are essentially distinct
 since this was true for $P_1(n,N),\, P_2(n,N),\, ...,P_k(n,N)$. It remains to 
 show that $P_i(n,N)$ and $P_j(n+h,N)$ are essentially distinct for any 
 $i,j=1,...,k$ provided $h$ is large enough and either $i\ne j$ or $i=j$ and
 $P_i(n,N)$ does not have the form $P_i(n,N)=p_in+Q_i(N)$. 
 Clearly, this is true if $P_i$ and $P_j$ have
 different degrees in $n$, and so we can assume that they have the same degree
 $d$ in $n$. Then we can write $P_i(n,N)=n^dV_i(N)+n^{d-1}W_i(N)
 +r_i(n,N)$ and $P_j(n,N)=n^dV_j(N)+n^{d-1}W_j(N)+r_j(n,N)$ where $V_i(N)$,
  $V_j(N)$ are nonzero while $W_i(N)$, $W_j(N)$ are arbitrary polynomials in
 $N$ only and $r_i(n,N)$, $r_j(n,N)$ are polynomials of degree less than
  $d-1$ in $n$. Then $P_j(n+h,N)=n^dV_j(N)+n^{d-1}(W_j(N)+dhV_j(N))+
  \tilde r_{j,h}(n,N)$ where $\tilde r_{j,h}(n,N)$ is a polynomial whose degree 
  in $n$ is less than $d-1$ having coefficients depending on $h$. Since
  $V_i(N)$ is a nonzero polynomial then for any $h$ large enough 
  $W_i(N)+dhV_i(N)\ne W_j(N)$ and if $d>1$ then $P_j(n+h,N)$ and $P_i(n,N)$
  are essentially distinct provided $h$ is large enough. The case $d=0$ is
  ruled out by our assumptions. If $d=1$ and $i\ne j$ then either $V_i\ne V_j$ 
  or $W_i\ne W_j$ and either $W_i$ or $W_j$ is nonconstant. In both of these
  cases $P_j(n+h,N)$ and $P_i(n,N)$ are essentially distinct. Next, if $d=1$
  and $i=j$ then $P_i(n+h,N)-P_i(n,N)=hV_i(N)$, and so $P_i(n+h,N)$ and
  $P_i(n,N)$ are essentially distinct if and only if $V_i$ is nonconstant
  concluding the proof of the lemma (where, in fact, we did not use that 
  $P_i$'s depend polynomially on $N$).
 \end{proof}
 
 Observe, that if deg$(\vf_i(n))\geq 2$, $\vf_i(n)=T_1^{P_{i1}(n)}\cdots
 T_k^{P_{ik}(n)}$ then $\max_{1\leq j\leq k}$deg$(P_{ij}(n))\geq 2$, and
 it follows from Lemma \ref{lem6.2} that $\vf_i(n+h)\vf_i^{-1}(h)$ depends
 nontrivially on $n$ provided $h$ is large enough. Rearranging $P$-polynomial
 expressions if needed, we can assume that deg$(\vf_i(n))=1$ for $i=1,...,q$
 and deg$(\vf_i(n))\geq 2$ for $i=q+1,...,k$. The condition deg$(\vf_i(n))=1$
 means that $P_{ij}(n)=p_{ij}n$ for some integers $p_{ij},\, j=1,...,k.$ Hence,
 in this case $\vf_i(n+h)=\vf_i(n)\vf_i(h)$. Thus, if $\Phi_i(n,N)=\vf_i(n)
 \psi_i(N)$ we can write 
 \begin{eqnarray*}
 & L(n,h)=\int\prod_{i=1}^q\Phi_i(n,N)(f_i\cdot\vf_i(h)f_i)\prod_{i=q+1}^k
 \Phi_i(n,N)f_i\\
 &\times\prod_{i=q+1}^k\Phi_i(n+h,N)\vf_i^{-1}(h)(\vf_i(h)f_i)d\mu=\int
 \prod_{i=1}^{k'}\tilde\Phi_i(n,N)\tilde f_id\mu
 \end{eqnarray*}
 where $k'=2k-q,\,\tilde f_i$ is either $f_l,\,\vf_l(h)f_l$ or it is 
 $f_l\cdot\vf_l(h)f_l$ for some $l$ between 1 and $k$ and $\tilde\Phi_l(n,N)$
 is either $\Phi_l(n,N)$ for some $l$ between 1 and $k$ or it is $\Phi_l(n+h,N)
 \vf_i^{-1}(h)$ for some $l$ between $q+1$ and $k$.
 
 Consider the new system $\tilde A_h=\{\Phi_i(n,N),\,\Phi_i(n+h,N)
 \vf_i^{-1}(h),\, i=1,...,k\}$ and suppose, without loss of generality, that 
 $\tilde\Phi_1(n,N)$ has the minimal weight in $\tilde A_h$. Since all
 $\vf_i(n)\ne \mbox{id}$ then $w(\tilde\Phi_1(n,N)$ is measure preserving and
 we can write 
 \begin{equation}\label{6.6}
 L(n,h)=\int\tilde f_1\cdot\prod_{i=2}^{k'}\hat\Phi_i(n,N)\tilde f_id\mu
 \end{equation}
 where $\hat\Phi_i(n,N)=\tilde\Phi_i(n,N)\tilde\Phi_1^{-1}(n,N)$. It follows 
 from the assumptions of Theorem \ref{thm2.8} that $\vf_i(n)\not\equiv\vf_l(n)$
 and $\vf_i(n+h)\not\equiv\vf_l(n+h)$ for $i,l=1,...,k,\, i\ne l$. Writing
 $\tilde\Phi_i(n,N)=\tilde\vf_i(n)\tilde\psi_i(N)$ we see from here and 
 Lemma \ref{lem6.2} that $\tilde\vf_i(n)\not\equiv\tilde\vf_l(n)$ for
 $i\ne l$ and large enough $h$. Writing $\hat\Phi_i(n,N)=
 \hat\vf_i(n)\hat\psi_i(N)$ we conclude from here that $\hat\vf_i(n)\not\equiv
 \mbox{id}$ and $\hat\vf_i(n)\not\equiv\hat\vf_l(n)$ for $i,l=2,...,k',\, i\ne l$
 for all $h$ large enough.
 
 Introduce the new system $A_h=\{\hat\Phi_i(n,N),\, i=2,...,k'\}$. In the
 same way as in \cite{BL} (refering the reader for more explanations there)
 we conclude that the weight matrix of $A_h$ precedes that of $A$. In order
 to invoke PET-induction we assume that Theorem \ref{thm2.8} holds true for
 all systems whose weight matrices precede that of $A$. Hence, we have for
 $A_h$,
 \begin{equation}\label{6.7}
 \bigg\|\frac 1N\sum_{n=0}^{N-1}\prod_{i=2}^{k'}\hat\Phi_i(n,N)\tilde f_i-
 \prod_{i=2}^{k'}\int\tilde f_id\mu\bigg\|_{L^2(X,\mu)}=\ve(N)\to 0
 \end{equation}
 as $N\to\infty$. Then by the Cauchy inequality
 \begin{eqnarray}\label{6.8}
 &\bigg\vert\frac 1N\sum_{n=0}^{N-1}L(n,h)-\prod_{i=1}^k\int\tilde f_id\mu
 \bigg\vert\\
 &=\big\vert\int\tilde f_1\big(\frac 1N\sum_{n=0}^{N-1}\prod_{i=2}^{k'}
 \hat\Phi_i(n,N)\tilde f_i-\prod_{i=2}^{k'}\int\tilde f_id\mu\big)d\mu\bigg\vert
 \leq\|\tilde f_1\|_{L^2(X,\mu)}\ve(N).\nonumber
 \end{eqnarray}
 Hence, by (\ref{6.6})--(\ref{6.8}),
 \begin{equation}\label{6.9}
 L(h)=\lim_{N\to\infty}\frac 1N\sum_{n=0}^{N-1}L(n,h)=\prod_{i=1}^{k'}\int
 \tilde f_id\mu.
 \end{equation}
 
 If one of $P_{ij}(n),\, j=1,...,k$ is not linear then deg$(\Phi_i(n,N))=
 $deg$(\vf_i(n))\geq 2$ and $\tilde f_l=f_k$ for some $l\leq k'$, and so 
 the last product in (\ref{6.6}) equals zero yielding 
 \begin{equation}\label{6.10}
 D-\lim_hL(h)=0.
 \end{equation}
 Otherwise, deg$(\Phi_i(n,N))=$deg$(\vf_i(n))=1$ for all $i$ and then $k'=k$,
 $\tilde f_i=f_i\cdot\vf_i(h)f_i$ and $\vf_i(n)=S_i^n$ for some $S_i\in G$,
 $S_i\ne \mbox{id}$. Then by weak mixing
 \[
 D-\lim_h\int f_i\cdot S_i^hf_id\mu=0
 \]
 which together with (\ref{6.9}) yields again (\ref{6.10}) concluding the 
 proof of Theorem \ref{thm2.8} since the initial step of the induction is
  given by (\ref{6.1}).                \qed

\subsection{Nonconvergence under weak mixing}\label{subsec6.2}
Next, we will show that, in general, weak mixing of $T$ is not enough
to ensure $L^2$-convergence in (\ref{1.3}) for general polynomials $P_j(n,N),\,
j=1,...,\ell$ taking on integer values on integers even in the "conventional"
case $\ell=1$. Consider the sum
\begin{equation}\label{6.11}
S_N=\sum_{n=1}^NT^{nN}f
\end{equation}
where $T$ is a measure preserving transformation of a separable probability
space $(X,\cB,\mu)$ and $f$ is a bounded measurable function. Recal, that
the Koopman operator $U_Tf(x)=f(Tx)$ is unitary and it has a spectral 
representation in the form
\begin{equation}\label{6.12}
U_T=\int_\Gam e^{2\pi iu}dE_u
\end{equation}
where $\{ e^{2\pi iu},\, u\in\Gam\}$ is the spectrum of $U_T$ and $E$ is
the corresponding projection operator valued spectral measure (see, 
for instance, \cite{Ha2} or \cite{Ru}). Then
\[
U_T^{nN}=\int_\Gam e^{2\pi iunN}dE_u,
\]
and so
\begin{equation}\label{6.13}
\| T^{nN}f-f\|_{L^\infty}\leq\sup_{u\in\Gam}|e^{2\pi iunN}-1|\leq
2\pi\sup_{u\in\Gam}\inf_{m\in\bbZ}|unN-m|.
\end{equation}

Fix a small $\ve>0$ and for each $M\in\bbN$ set 
\[
\Gam_{\ve,M}=\{ u:\,\inf_{m\in\bbZ}|uM-m|\leq\ve\}.
\]
Observe that if $u\in\Gam_{\ve,N}$ then $nu\in\Gam_{n\ve,N}$ and 
$\Gam_{\ve,N}\subset\Gam_{n\ve,nN}$. Define inductively $N_0=1$ and
$N_{k+1}=[\frac {5N^2_k}\ve],\, k=0,1,2,...$ where $[a]$ is the integral
part of $a$. Set also $\ve_k=\frac \ve{N_k},\, k=0,1,...$. Then
\begin{equation}\label{6.14}
\Gam_{\ve,nN_k}\subset\Gam_{\ve_k,N_k}\quad\mbox{for all}\,\, n=1,2,...,N_k
\end{equation}
and $\Gam_\ve=\bigcap_{k=1}^\infty\Gam_{\ve_k,N_k}$ is a Cantor like set,
in particular, it is a perfect set and for any $k$,
\begin{equation}\label{6.15}
\max_{1\leq n\leq N_k}\sup_{u\in\Gam_\ve}|e^{2\pi iunN_k}-1|\leq 2\pi\ve.
\end{equation}

Let $\nu_\ve$ be a continuous (non-atomic) probability measure on $\Gam_\ve$,
say, constructed in the same way as the Cantor distribution on the standard
Cantor set. Next, we introduce a spectral measure $E^{(\ve)}$ concentrated
on $\Gam_\ve$ by the standard formula $E^{(\ve)}_Ug=\bbI_Ug$ for each 
measurable function $g$ on $\Gam_\ve$ and a measurable set $U\subset\Gam_\ve$
where $\bbI_U$ is the indicator of $U$. The spectral measure $E^{(\ve)}$ is
continuous considering it on the probability space $(\Gam_\ve,\nu_\ve)$ since
for each $u\in\Gam_\ve$ any function $\bbI_{\{ u\}}g$ is zero $\nu_\ve$-almost
everywhere. Next, we can find a transformation $T$ such that its Koopman
operator $U_T\vf=T\vf$ has the spectral representation
\[
U_T=\int_{\Gam_\ve}e^{2\pi iu}dE_u^{(\ve)}
\]
(see, for instance, Ch. 4 in \cite{EFHN}) and since $E^{(\ve)}$ is a continuous
spectral measure then $T$ is weakly mixing (see, for instance, \cite{Ha1} or
\cite{Wa}).

By (\ref{6.15}),
\begin{equation}\label{6.16}
\| T^{nN_k}f-f\|_{L^2}\leq 2\pi\ve\| f\|_{L^2}
\end{equation}
for any $n=1,2,...,N_k$, and all $k=1,2,...$. Hence
\[
\| f\|_{L^2}-\|\frac 1{N_k}S_{N_k}\|_{L^2}\leq
\|\frac 1{N_k}S_{N_k}-f\|_{L^2}\leq 2\pi\ve\| f\|_{L^2}.
\]
Now, choose a function $f$ such that $\int fd\mu=0$ and $\int|f|d\mu>0$. 
If the $L^2$ ergodic theorem holds true for the averages $\frac 1NS_N$ then
$\|\frac 1{N_k}S_{N_k}\|_{L^2}\to 0$ as $k\to\infty$ which leads to the 
contradiction in the above inequality if $\ve<\frac 1{2\pi}$.   \qed

\subsection{Proof of Theorem \ref{thm2.9}}\label{subsec6.3}
For the proof of Theorem \ref{thm2.9} we will need the following result.
\begin{lemma}\label{lem6.3} Let $P(n,N)$ be a nonconstant polynomial of $n$
and $N$ taking on integer values on integers. Set
\[
M_K(N)=|\{ 1\leq n\leq N:\, |P(n,N)|\leq K\}|
\]
where $|\{\cdot\}|$ denotes the cardinality of a set in brackets and if
$P(n,N)=P(N)$ does not depend on $n$ then we set $M_K(N)=N$ if $|P(N)|\leq K$
and $M_K(N)=0$, for otherwise. If $P(n,N)$ nontrivially depends on $n$ then
\begin{equation}\label{6.17}
M_K(N)\leq (2K+1)\mbox{deg}_nP
\end{equation}
where deg$_n$ is the degree of the polynomial in $n$ considering $N$ as a 
constant. If $P(n,N)=P(N)$ depends only on $N$ then there exists $N_0$ such
that $|P(N)|>K$ for all $N\geq N_0$, and so $M_K(N)=0$ for such $N$. In both
cases $\lim_{N\to\infty}\frac 1NM_K(N)=0$.
\end{lemma}
\begin{proof} For any $k=0,\pm 1,\pm 2,...,\pm K$ there exists at most 
deg$P$ solutions in $n$ of the equation $P(n,N)=k$, and so (\ref{6.17})
follows. If $P(n,N)=P(N)$ is nonconstant then $|P(N)|\to\infty$ as $N\to\infty$
and the second assertion follows, as well.
 \end{proof}
 
 Next we can prove Theorem \ref{thm2.9}. As before, without loss of generality
 we can assume that, at least, one of functions $f_j$ has zero integral with 
 respect to $\mu$. Set
 \[
 x_{n,N}=\prod_{j=1}^\ell T^{P_j(n,N)}f_j
 \]
 and in order to prove Theorem \ref{thm2.9} we have to show that
 \begin{equation}\label{6.18}
 \lim_{N\to\infty}\|\frac 1N\sum_{n=1}^Nx_{n,N}\|_{L^2}=0.
 \end{equation}
 which according to Lemma \ref{lem6.1} will follow if (\ref{6.2}) holds true.
 
 Without loss of generality assume that $1,2,...,k,\, k\leq\ell$ are all 
 indexes $j$ such that $P_j(n,N)=p_jn+Q_j(N)$ for some nonzero integers
 $p_j$ and polynomials $Q_j$ in $N$ taking on integer values on integers.
 Then
 \begin{eqnarray*}
 &\langle x_{n,N},x_{n+h,N}\rangle=\int\prod_{j=1}^\ell T^{P_j(n,N)}f_j
 \prod_{j=1}^\ell T^{P_j(n+h,N)}f_jd\mu\\
 &=\int\prod_{j=1}^kT^{p_jn+Q_j(N)}(f_jT^{p_jh}f_j)\prod_{j=k+1}^\ell
 T^{P_j(n,N)}f_j\prod_{j=1}^\ell T^{P_j(n+h,N)}f_jd\mu.
 \end{eqnarray*}
 By Lemma \ref{lem6.2}, $P_1(n,N),...,P_\ell(n,N);P_{k+1}(n+h,N),...,
 P_\ell(n+h,N)$
 are essentially distinct polynomials, and so their pairwise differences
 $p_{ij}^{(1)}(n,N)=P_i(n,N)-P_j(n,N),\, p_{ij}^{(2)}(n+h,N)=P_i(n+h,N)-
 P_j(n+h,N),\, i,j=1,...,\ell,\, i\ne j$ and $p_{ij}^{(3)}(n,N)=P_i(n,N)-
 P_j(n+h,N),\, i=1,...,\ell,\, j=k+1,...,\ell$ are nonconstant polynomials 
 of $n$ and $N$. Since $T$ is strongly $2\ell$-mixing then for any $\ve>0$ 
 and any bounded measurable functions $g_1,...,g_L$ with $L\leq 2\ell$
 there exists $K_\ve>0$ such that 
 \[
 |\int\prod_{j=1}^{L}T^{m_j}g_jd\mu-\prod_{j=1}^L\int g_jd\mu|<\ve\,\,
 \mbox{provided}\,\,\min_{1\leq i,j\leq L,i\ne j}|m_i-m_j|>K_\ve.
 \]
  By Lemma \ref{lem6.3},
 \[
 \lim_{N\to\infty}\frac 1N\big\vert\bigcup_{l=1}^3\bigcup_{i,j}\{ 1\leq n
 \leq N:\, |p_{ij}^{(l)}(n,N)|\leq K_\ve\}\big\vert=0
 \]
 where $i,j$ run over indexes appearing in the above definitions of
 $p_{ij}^{(l)}$'s.
 Hence, for $h$ large enough choosing $K_\ve$ for functions $g_j$ equal
 either to some $f_l$ or to $f_lT^{p_lh}f_l$ we obtain,
 \begin{eqnarray*}
 &\limsup_{N\to\infty}\big\vert\frac 1N\sum_{n=1}^N\langle x_{n,N},\, 
 x_{n+h,N}\rangle\\
 &-\prod_{j=1}^k\int f_jT^{p_jh}f_jd\mu\prod_{j=k+1}^\ell
 (\int f_jd\mu)^2\big\vert<\ve
 \end{eqnarray*}
 and since $\ve>0$ is arbitrary we obtain that
\begin{eqnarray*}
&\lim_{N\to\infty}\frac 1N\sum_{n=1}^N\langle x_{n,N},\, 
x_{n+h,N}\rangle\\
&=\prod_{j=1}^k\int f_jT^{p_jh}f_jd\mu\prod_{j=k+1}^\ell
 (\int f_jd\mu)^2.
 \end{eqnarray*}
 Finaly, relying on strong mixing we let $h\to\infty$ and obtain
 \[
 \lim_{h\to\infty}\lim_{N\to\infty}\frac 1N\sum_{n=1}^N\langle 
 x_{n,N},\, x_{n+h,N}\rangle =\prod_{j=1}^\ell(\int f_jd\mu)^2=0
 \]
 since one of integrals $\int f_jd\mu$ is zero, completing the proof of
 Theorem \ref{thm2.9}.   \qed

\bibliography{matz_nonarticles,matz_articles}
\bibliographystyle{alpha}

\end{document}